\documentclass[12pt]{amsart}
\usepackage{amssymb}
\usepackage{amsfonts}
\usepackage{amsmath}
\usepackage{graphicx}
\usepackage{color}
\usepackage{amsmath}
\usepackage{stmaryrd}
\usepackage{epsfig,color}
\usepackage{blindtext}
\usepackage{enumerate}
\usepackage{hyperref}
\usepackage{url}
\usepackage{bbm}
\usepackage{filecontents}
\usepackage{nicefrac,mathtools}
\DeclareGraphicsExtensions{.pdf,.jpeg,.png}
\usepackage{epstopdf}
\usepackage{cancel} 
\usepackage[normalem]{ulem} 
\usepackage{verbatim}		

\usepackage{color}
\usepackage[msc-links, lite]{amsrefs}
\usepackage{geometry}
\newtheorem{theorem}{Theorem}[section]

\newtheorem{proposition}[theorem]{Proposition}
\newtheorem{lemma}[theorem]{Lemma}
\newtheorem{corollary}[theorem]{Corollary}

\newtheorem{claim}[]{Claim}

\theoremstyle{definition}
\newtheorem{definition}[theorem]{Definition}
\theoremstyle{remark}
\newtheorem{remark}[theorem]{Remark}

\def\m{\mathbf m}
\def\n{\mathbf n}

\def\wti{\widetilde}
\def\mb{\mathbb}

\def\ric{\mathrm{Ric}}
\def\area{\mathrm{Area}}

\def\dist{\mathrm{dist}}

\newcommand{\mc}{\mathcal}
\newcommand{\mf}{\mathbf}
\newcommand{\pps}{\frac{\partial}{\partial s}}
\newcommand{\dv}{\mathrm{div}}

\numberwithin{equation}{section}

\title[Degenerate free boundary minimal hypersurfaces]{Compactness and generic finiteness for free boundary minimal hypersurfaces (II)}
\date{\today}

\author{Zhichao Wang}
\address{Max-Planck Institute for Mathematics, Vivatsgasse 7, 
53111 Bonn, Germany}
\email{wangzhichaonk@gmail.com}

\begin{document}

\begin{abstract}
Given a compact Riemannian manifold with boundary, we prove that the limit of a sequence of embedded, almost properly embedded free boundary minimal hypersurfaces, with uniform area and Morse index upper bound, always inherits a non-trivial Jacobi field. To approach this, we prove a one-sided Harnack inequality for minimal graphs on balls with many holes.
\end{abstract}
\maketitle

\section{Introduction}
\subsection{Main results}
Let $(M^{n+1},\partial M,g)$ be a compact Riemannian manifold with boundary of dimension $3\leq (n+1)\leq 7$. An $n$-submanifold $\Sigma$ is a critical point of the $n$-dimensional area functional if and only if the mean curvature of $\Sigma$ vanishes everywhere and $\Sigma$ meets $\partial M$ orthogonally. Such $n$-submanifolds are called {\em free boundary minimal hypersurfaces} (see Definition \ref{def:fbmh}).

Given a free boundary minimal hypersurface $\Sigma$, the second variation of area functional produces discrete eigenvalues and eigenfunctions in $C^\infty(\Sigma)$. Then the dimension of the maximal subspace of $C^\infty(M)$ that the second variation is negative definite, is called the {\em index} of $\Sigma$, denoted by $\mathrm{index}(\Sigma)$ (cf. \S \ref{sec:pre}). And the eigenfunctions corresponding to the zero eigenvalue are called the {\em Jacobi fields} (see Definition \ref{def:jacobi field}).

Denote by $\mc M(\Lambda,I)$ the space of embedded free boundary minimal hypersurfaces with $\area\leq \Lambda$ and $\mathrm{index}\leq I$.

The compactness of $\mc M(\Lambda,0)$ was firstly studied by Fraser-Li \cite{FL} for $3$-manifolds with non-negative Ricci curvature and convex boundary, and Guang-Li-Zhou \cite{GLZ16} for higher dimensions without curvature assumptions.

Recently, Ambrozio-Carlotto-Sharp \cite{ACS17} proved the compactness of $\mc M(\Lambda, I)$ under additional assumptions. Moreover, they proved that the limit hypersurface has non-trivial Jacobi fields. Later, the compactness result has been proved in \cite{GZ18} for all compact Rimemannian manifolds with boundary. The degeneration of the limit hypersurface has also been obtained when the convergence has multiplicity one. In this paper, we enhance this theorem by considering the case of higher multiplicity in convergence.
\begin{theorem}\label{thm:main thm}
Let $\{\Sigma_k\}\subset \mc M(\Lambda,I)$ locally smoothly converges to $\Sigma\in \mc M(\Lambda,I)$ with multiplicity $m$. Suppose that $m\geq 2$. Then $\Sigma$ has a positive Jacobi field.
\end{theorem}

The first study of such compactness was due to Choi-Schoen \cite{CS85}, who proved compactness for minimal surfaces with bounded topology in closed three-manifolds with positive Ricci curvature. In higher dimensions, Schoen-Simon-Yau \cite{SSY} and Schoen-Simon \cite{SS} proved interior curvature estimates and compactness for stable closed minimal hypersurfaces with uniform area upper bound. 

Their results were later generalized by Sharp \cite{Sha17} to minimal hypersurfaces with uniform Morse index and area upper bound, which also says that the limit hypersurface is {\em degenerate}, i.e. has non-trivial Jacobi fields. Combining with the Bumpy Metric Theorem given by White \cite{Whi91}, there are only finitely many embedded minimal hypersurfaces with uniform Morse index and area upper bound in manifolds with bumpy metrics. Such results plays an important role in the index estimates of minimal hypersurfaces using min-max construction, proved by Marques-Neves \cite{MN16}.

As a direct consequence of Theorem \ref{thm:main thm}, we obtain the following generic finiteness theorem for free boundary minimal hypersurfaces.
\begin{corollary}
Let $M^{n+1}$ be a compact manifold with boundary and $3\leq(n+1)\leq 7$. Fix $I\in\mb N$ and $\Lambda> 0$. Then for a generic metric on M , there are only finitely many almost properly embedded free boundary minimal hypersurfaces in $\mc M(\Lambda,I)$.
\end{corollary}

\begin{remark}
We can compare the results between minimal hypersurface with free boundary and closed cases. Fraser-Li's result \cite{FL} is a natural free boundary analog of Choi-Schoen’s result \cite{CS85}; Guang-Li-Zhou \cite{GLZ16} obtained the free boundary version of Schoen-Simon-Yau \cite{SSY} and Schoen-Simon's results \cite{SS}; \citelist{\cite{ACS17},\cite{GZ18}} and Theorem \ref{thm:main thm} together can be seen as a generalization of \cite{Sha17}.
\end{remark}

\subsection{Harnack inequality}\label{subsec:intro:harnack inequalty}
We approach Theorem \ref{thm:main thm} by proving a Harnack inequality. Let $\mc N$ be a minimal hyersurface in $M$. Denote by $\mc B(p;r)$ the geodesic ball in $\mc N$. Denote by $\mc A(p;r,s)=\mc B(p;s)\setminus\mc B(p;r)$. Similarly, the geodesic ball in $M$ is denoted by $B(p;r)$.

Let $\Gamma$ and $\Sigma$ be positive minimal graphs with functions $v,u$ on $\mc A(p;r^2,2\epsilon^2)\setminus\cup_{j=1}^I\mc B(q_j;r^2)$ satisfying
\begin{equation}\label{eq:intro:boundary barrier}
0<v(x)-u(x)\leq C_1(|x|^2+r^2), \ \ \text{ where } |x|=\dist_\Sigma(x,p).
\end{equation}
Then the key ingredient to approach Theorem \ref{thm:main thm}, roughing speaking, is to prove the following:
\begin{theorem}[Theorem \ref{thm:harnack for minimal graphs on minimal surfaces}]\label{thm:intro:harnack}
There exists $C=C(M,\mc N,C_1,I)$, $\epsilon_0=\epsilon_0(M,\mc N,C_1,I)$ so that if $\epsilon<\epsilon_0$, $\{q_j\}_{j=1}^I\subset\mc B(p;\epsilon r)$, then
\begin{equation*}
\max_{\partial \mc B(p;\epsilon^2)}(v-u)(x)\leq C\min_{\partial \mc B(p;2{\epsilon}r)}(v-u)(x).
\end{equation*}
\end{theorem} 

In high dimensional cases, i.e. $4\leq (n+1)\leq 7$, we have the following better inequality, which can deduce Theorem \ref{thm:intro:harnack} directly.
\begin{theorem}[Theorem \ref{thm:one sided harnark for high dimension}]\label{thm:intro:one sided harnark for high dimension}
Let $\Gamma,\Sigma$ be positive minimal graphs with functions $v,u$ on $\mc A(p;r,2R)$ satisfying
\[
0<v(x)-u(x)\leq C_1|x|^2,\  \forall x\in \mc A(p;r,2R), \text{\ where\ } |x|=\dist_\Sigma(x,p).\]
There exists $C$ and $\epsilon_0$ depending only on $C_1,M,\mc N$ so that if $R\leq \epsilon_0$, then
\begin{equation*}
\max_{\partial \mc B(p;R)}(v-u)(x)\leq C\min_{\partial \mc B(p;2r)}(v-u)(x).
\end{equation*}
\end{theorem} 

We remark that in three-dimensional case, Theorem \ref{thm:intro:harnack} is sharp in some sense, which means that it is impossible get the estimate in Theorem \ref{thm:intro:one sided harnark for high dimension}. For the height of Catenoid in $\mathbb R^3$ tends to infinity even if it is small over $\mc A(0;r,4r)$.

The proofs of Theorem \ref{thm:intro:harnack} and \ref{thm:intro:one sided harnark for high dimension} are very technical and hence occupy the most pages of this paper (see \S \ref{sec:high dim} and \S \ref{sec:three dim}). However, the idea is quite clear. For $\epsilon$ small enough and the assumption (\ref{eq:intro:boundary barrier}), the graph function can be seen as an `almost harmonic function' on $\Xi$. By scaling $\Xi$ to a normal size, it can be regarded as a subset of Eculidean space. Then the Harnack inequality looks natural. 

The difficulty here is that we can not use such blow-up argument directly because there are no suitable scaling size to make $\epsilon$ to be finite and $r^2$ to be positive simultaneously.

The classical methods from PDE to produce Harnack inequalities does not work since we have many boundaries here. Note that the classical Harnack says the maximum is bounded by the minimum in the interior. However, the Harnack inequality in Theorem \ref{thm:intro:harnack} only states that the value of outside boundary can be bounded by that of inside boundary (so called {\em one-sided Harnack inequality}). Namely, the opposite inequality does not holds true by considering the Catenoid in $\mb R^3$.  

Due to the so many boundaries inside, we can not use the minimal foliation argument given by White \cite{Whi87} to obtain the Harnack as in \cite{Sha17}. 

Therefore, we approach Theorem \ref{thm:intro:harnack} by studying the differential inequality directly.

\subsection{Outline of the proof of Theorem \ref{thm:main thm}}
We first recall the argument in \cite{GZ18}. Given a sequence of $\Sigma_k\in\mc M(\Lambda, I)$, then there exists $\Sigma\in\mc M(\Lambda,I)$ and a finite set $\mc W\subset\Sigma$ with $\#\mc W\leq I$ so that $\Sigma_k$ locally smoothly converges to $\Sigma$ in $M\setminus\mc W$ with multiplicity $m$. Hence $\Sigma_k$ can be regarded as multi-graph on $\Sigma\setminus \mc W$ with graph function 
\[u^1<u^2<...<u^m.\]
Then inspired by Simon \cite{Sim87}, the difference of top and bottom sheet may converges to a Jacobi field $w$ with possibly singular point on $\mc W$. Then the aim is to prove that $\mc W$ are all removable singular set.

Comparing to \cite{ACS17}, the difficulty is that $\mc W$ may have touching set of $\Sigma$, i.e. the set in $\Sigma\cap \partial M\setminus \partial \Sigma$. Let $p\in\mc W$ be a touching point of $\Sigma$. Assume that $\partial M$ is on the non-positive side of $\Sigma$ near $p$. Then \cite{GZ18}*{Claim D} says that for $\epsilon$ small enough, 
\[\max_{\partial\mc B(p;r)}|u^m|\leq C\max_{\partial\mc B(p;\epsilon)}u^m.\]
This gives a removable singularity theorem for the limit of the normalization of $u^m$. Such a theorem is not enough to prove the existence of entire Jacobi fields because the top sheet near two singularities may not be the same. To overcome this, we need to prove that normalization of $u^1$ also converges to a smooth function.

We argue it by contradiction. Suppose not, then by the Harnack inequality on $\partial \mc B(p;r)$ obtained in \cite{GZ18}, the normalization of $u^1$ tends to $-\infty$ at $p$. So we can take $\epsilon\ll 1$ so that $h\geq\kappa u^m$ on $\partial\mc B(p;\epsilon)$ for $\kappa\gg 1$, where $h$ is the minimum of $-u^1$ on $\partial\mc B(p;\epsilon)$. Now let $S_k$ be the subset of $\Sigma_k$ near $\mc B(p;\epsilon)$ such that the $\Sigma_k$ intersects with the level set of $\Sigma$ with large angles. Since $\mathrm{index}(\Sigma)\leq I$, then we need at most $I$ balls $B(q;\rho(q))$ with $q\in S_k$ and $\rho(q)=L(|q|^2+h/\kappa)$ (see Claim \ref{claim:I balls cover Sk}). Denote by these balls $\{B(q_j;\rho(q_j))\}$.   

Let $\Sigma_k'$ be the component of $\Sigma_k\setminus \bigcup_jB(q_j;\rho(q_j))$ containing the bottom sheet over $\partial\mc B(p;\epsilon)$. Then $\Sigma_k'$ can be seen as a minimal graph over $\mc B(p;\epsilon)$. After applying Theorem \ref{thm:intro:harnack} at most $I$ times, we obtain
\[ \max_{\partial\mc B(p;\epsilon)}-u^1\leq C\min_{\partial\mc B(p;C_0I\sqrt {h/\kappa})}-u^1\leq Ch/\kappa,\]
which leads to a contradiction for $\kappa$ large enough.

To proceed the argument of Theorem \ref{thm:main thm}, it suffices to prove Theorem \ref{thm:intro:harnack}. Here we give an outline of the proof of it. Denote by
\[\tau(q;s)=\int_{\partial\mc B(q;s)}\langle \nabla w,\nu\rangle, \ \ \text{ and }\ \ \mc I(q;s)=s^{1-n}\int_{\partial\mc B(q;s)}w,\]
where $w=v-u$. Then by the Harnack inequality obtained in \cite{GZ18} (see also Corollary \ref{cor:harnack for equivalent radius}), $\mc I(q;s)$ can be seen as the value of $w$ on $\partial\mc B(q;s)$. 

A direct computation (see (\ref{eq:three dim:from R^2 to R})) by divergence theorem gives that 
\[\mc I(p;\epsilon)\leq 2\mc I(p;2\sqrt\epsilon r)+2\tau_0|\log (\sqrt\epsilon/r)|,\]
where $\tau_0=\tau (p;2\sqrt\epsilon r)$. Hence without loss of generality, we assume $\mc I(p;2 \sqrt\epsilon r)\leq \tau_0|\log(\sqrt\epsilon/r)|$. Then we can find $y_1\in Q$ ($:=\{q_i\}_{i=1}^I$ in \S \ref{subsec:intro:harnack inequalty}) and $\theta_1\in (1/4^{4I+3},1/16)$ (see Step A in Proposition \ref{prop:good sequence of Q from R to R3/4}) so that $\tau (q_j;\theta_1|q_j|)\geq 1/4^{I+3}$ and
\[\mc I(p; r)-c_0\mc I(y_1;\theta_1 |y_1|)\geq c_0(\log r-\log (\theta_1 |y_1|)).\]
Repeating the argument above, we can find a sequence of $\{y_j\}\subset Q$ so that 
\begin{align*}
&\mc I(p; \theta_j |y_j-y_{j-1}|)-c_0\mc I(y_{j+1};\theta_{j+1} |y_{j+1}-y_j|)\\
&\geq c_0\Big[\log(\theta_{j}|y_j-y_{j-1}|)-\log (\theta_{j+1} |y_{j+1}-y_j|)-c_1\Big].
\end{align*}
By adding them together with suitable coefficients (see Lemma \ref{lem:estimate of taulogR}), we obtain 
\[\mc I(p;\sqrt\epsilon r)\geq c|\log (\sqrt\epsilon/r)|.\]
Then the desired results follows.

\vspace{1em}
This paper is organized as follows: in Section \ref{sec:pre}, we will first give some notations; and in Section \ref{sec:harnack}, we state some Harnack inequalities, including the classical one from blowing-up arguments and our new one-sided one; Using these, we construct Jacobi fields in Section \ref{sec:to free boundary minimal hypersurface}; The proof of One-sided Harnack inequality is in Section \ref{sec:high dim} for $n\geq 3$ and Section \ref{sec:three dim} for $n=2$, we give a proof of Harnack inequality in high dimensions; after that, some lemmas and tedious computation will be displayed in Appendix \ref{sec:appendix:minimal graph function} \ref{sec:appendix:der lemma} \ref{sec:appendix:rearrange} and \ref{sec:appendix:connectedness}.

\subsection*{Acknowledgment:} I would like to thank Prof. Xin Zhou for bringing this problem to us and many helpful discussion, and thank Prof. Minicozzi for bringing our attention to the paper of \cite{CM02}. I would also like to thank Qiang Guang for reading the preprint and Weiming Shen for many helpful discussion on harmonic functions.

\section{Preliminaries}\label{sec:pre}
In this section, we collect some basic definitions and preliminary results for free boundary minimal hypersurfaces. We refer to \cite{GZ18} for detailed notions.

Let $M^{n+1}$ be a smooth compact Riemannian manifold with non-empty boundary $\partial M$. We may assume that $M\hookrightarrow \mb{R}^L$ is isometrically embedded in some Euclidean space. By choosing $L$ large, we assume that $M$ is a compact domain of a closed $(n+1)$-dimensional manifold $\wti{M}$.

Let $\Sigma^n$ be a smooth $n$-dimensional manifold with boundary $\partial \Sigma$ (possibly empty). A smooth embedding $\phi: \Sigma \to M$ is said to be an \emph{almost proper embedding} of $\Sigma$ into $M$ if  $\phi(\Sigma)\subset M$ and  $\phi(\partial \Sigma)\subset\partial M$. We  write $\Sigma=\phi(\Sigma)$ and $\partial \Sigma=\phi(\partial \Sigma)$.

We use $\mathrm{Touch}(\Sigma)$ to denote the touching set $\mathrm{int}(\Sigma)\cap \partial M$. If the touching set $\mathrm{Touch}(\Sigma)$ is empty, then we say that $\Sigma$ is \emph{properly embedded}.

\begin{definition}\label{def:fbmh}
An almost properly embedded hypersurface  $(\Sigma, \partial \Sigma)\subset (M,\partial M)$ is called a {\em free boundary minimal hypersurface} if and only if 
the mean curvature of $\Sigma$ vanishes and $\Sigma$ meets $\partial M$ orthogonally along $\partial \Sigma$. 
\end{definition}

Let $\Sigma^n\subset M^{n+1}$ be an almost properly embedded free boundary minimal hypersurface. The quadratic form of $\Sigma$ associated to the
second variation formula is defined as
\[Q(v,v)=\int_\Sigma \left(|\nabla^\perp v|^2- \mathrm{Ric}_M(v,v)-|A^\Sigma|^2|v|^2 \right)\,d\mu_\Sigma  - \int_{\partial \Sigma} h^{\partial M}(v,v)\,d\mu_{\partial \Sigma},
\]
where $v$ is a section of the normal bundle of $\Sigma$, $\mathrm{Ric}_M$ is the Ricci curvature of $M$, $A^\Sigma$ and $h$ are the second fundamental forms of the hypersurfaces $\Sigma$ and $\partial M$, respectively.

The \emph{Morse index of $\Sigma$ on the proper subset $\Sigma\setminus\mathrm{Touch}(\Sigma)$} is defined to be the maximal dimension of a linear subspace of sections of normal bundle $N\Sigma$ compactly supported in $\Sigma\setminus\mathrm{Touch}(\Sigma)$ such that the quadratic form $Q(v,v)$ is negative definite on this subspace.

\begin{remark}
In the following of this paper, the `Morse index of $\Sigma$' always means the `Morse index on the proper subset $\Sigma\setminus\mathrm{\partial \Sigma}$', denoted by $\mathrm{index}(\Sigma)$.
\end{remark}

\begin{definition} An almost properly embedded free boundary minimal hypersurface $\Sigma^n\subset M$ is said to be \emph{stable away from the touching set} $\mathrm{Touch}(\Sigma)$ if the Morse index of $\Sigma$ is 0.
\end{definition}

\begin{definition}\label{def:jacobi field}
We say that a function $f\in C^\infty(\Sigma)$ is a Jacobi field of $\Sigma$ is $f$ satisfies
\begin{equation}\label{equ:jacobi}
\left\{
\begin{array}{ll}
\Delta_\Sigma f + (\mathrm{Ric}_M(\mf{n},\mf{n}) +|A^\Sigma|^2)f=0\quad  & \text{on}\, \Sigma,\\
\frac{\partial f}{\partial \eta}=h^{\partial M}(\mf{n},\mf{n})f \quad & \text{on}\, \partial \Sigma,
\end{array}
\right.
\end{equation}
where $\eta$ is the co-normal of $\Sigma$.
\end{definition}

For simplicity, we will use $\mc{M}(\Lambda,I)$ to denote the set of almost properly embedded free boundary minimal hypersurfaces with $\area\leq \Lambda$ and $\mathrm{index}(\Sigma)\leq I$.

We remark that in the proofs of our results, we often allow a constant $C$ to change from line to line, and the dependence of $C$ should  be clear in the context.

\section{Harnack inequalities for minimal graphs}\label{sec:harnack}
In this section, $M^{n+1}$ is always a closed manifold and $\mc N$ is an embedded compact minimal hypersurface in $M$ so that $\mc B(p;1)\cap \partial N=\emptyset$, where $\mc B(p;r)$ the intrinsic geodesic ball of $\mc N$ with radius $r$ and center $p\in \mc N$. 

\subsection{The Harnack inequalties on geodesic spheres}
In this subsection, we always assume $3\leq(n+1)\leq 7$.

We recall Harnack estimate in a disk first. And then we state an second order estimate for minimal graph functions.
\begin{lemma}[Gradient estimates,\cite{GZ18}*{Lemma 6.1}]\label{lem:gradient estimates}
Suppose that two sequences of embedded compact minimal graphs (over $\mc B(p;1)$) $\{\Sigma_k\}$ and $\{\Gamma_k\}$ with graph functions $\{u_k\}$ and $\{v_k\}$ converge smoothly to $\mc B(p;1)$ and $u_k-v_k\geq 0$. Then there exists a constant $C=C(M,\mc N)$ such that for any $r>0$ and $p'\in\mc B(p;1-r)$, we have
\[\limsup_{k\rightarrow\infty}\sup_{x\in\mc B(p';r)}(r-\dist_\mc N(x,p'))|\nabla\log(u_k-v_k)(x)|\leq C.\]
\end{lemma}

\begin{definition}\label{def:fK pair}
Let $\Omega\subset \mc B(p;1)$ be an open set and $u, v$ be two functions on $\Omega$. Given a constant $K>0$ and a positive function $f\in C^0(\Omega)$, we say that $(v,u)$ is {\em a $(f,K)$-pair} if they satisfy
\begin{equation}\label{eq:der:graph assumptions}
v(x)-u(x)>0,\ \ |u(x)|+|v(x)|<f,\ \ |\nabla u(x)|+|\nabla v(x)|<K.
\end{equation}
We say $(v,u)$ is a {\em strong $(f,K)$-pair} if it is a $(f,2)$-pair and 
\[  |\nabla v(x)|+|\nabla^2v(x)|\leq K|v(x)|\leq K^2|x|.\]
\end{definition}

Furthermore, we have the following estimates:
\begin{lemma}\label{lem:der estimate}
Let $M^{n+1}$ be a closed manifold with $3\leq (n+1)\leq 7$ and $\mc N$ be an embedded compact minimal hypersurface in $M$ so that $\mc B(p;1)\cap \partial\mc N=\emptyset$. Given a constant $K>0$, there exist constants $C=C(M,\mc N,K)$ and $\delta=\delta(M,\mc N,K)$ so that if $q\in \mc B(p;1-r)$ for some $0<r<1$, $\Sigma$ and $\Gamma$ are minimal graphs with graph functions $u$ and $v$ over $\mc B(q;r)\setminus V$ for a compact subset $V\subset\mc N$ and $(v,u)$ is a $(\delta,K)$-pair,
then
\begin{gather*}
\dist_\mc N(x,\partial\mc B(q;r)\cup V)\cdot|\nabla\log(v-u)(x)|<C,\\
\dist^2_\mc N(x,\partial\mc B(q;r)\cup V)\cdot\frac{|\nabla^2(v-u)(x)|}{(v-u)(x)}<C.
\end{gather*}
\end{lemma}

This lemma can be proved by a standard blow-up process, which is the same with Lemma \ref{lem:gradient estimates}. We give the proof in Appendix \ref{sec:appendix:der lemma} for the completeness of this paper.

\begin{corollary}\label{cor:harnack for equivalent radius}
Given a constant $K>0$, $\theta\in(0,1/8)$ and $0<R<1/2$, there exist constants $C=C(M,\mc N,K,I,\theta)$, $C_0=C_0(M,\mc N,\theta)$ and $\delta=\delta(M,\mc N,K)$ so that if $\Sigma$ and $\Gamma$ are minimal graphs with functions $u$ and $v$ over $\mc A(p;\theta R,2R)\setminus \bigcup_{j=1}^I\mc B(q_j;r)$, and $(v,u)$ is a $(\delta,K)$-pair (see Definition \ref{def:fK pair}) and 
\begin{itemize}
\item $\mc B(q_j;r)\subset\subset \mc A(p;4\theta R,R/2)$;
\item $\theta R\geq C_0Ir$,
\end{itemize}	
then we have	
\[\max_{x\in \partial \mc B(p;R)}(v-u)(x)\leq C\min_{x\in\partial \mc B(p;2\theta R)}(v-u)(x).\]
\end{corollary}
\begin{proof}
For simplicity, denote by $w(x)=u(x)-v(x)$.

By Lemma \ref{lem:enlarge interior radius}, we can take $C_0$ suitable so that there exists a $C^1$ curve $\gamma:[0,1]\rightarrow \mc A(p;2\theta R,R)$ connecting $\partial \mc B(p;2\theta R)$ and $\partial \mc B(p;R)$ so that
\begin{equation}\label{eq:weak harnack:suitable curve}
 \mathrm{Length}(\gamma)\leq C_0 R \text{\ \  and\ \  } \dist(\gamma,\cup_j\mc B(q_j;r))\geq \theta R/(C_0I).
\end{equation} 
Then Lemma \ref{lem:der estimate} gives that there exists $C_1=C_1(M,\mc N,K)$ so that for any $x\in\gamma$,
\[|\nabla\log w(x)|\leq  C_1/\dist(x,\cup_j\mc B(q_j;r)).\]
Integrating it over $\gamma$, together with (\ref{eq:weak harnack:suitable curve}) we have
\[w(\gamma(0))\leq e^{C_0^2C_1I/\theta} w(\gamma(1)).\] 
Moreover, the Lemma \ref{lem:der estimate} also implies that 
\[\max_{\partial\mc B(p;R)}w\leq e^{C_1}\min_{\partial\mc B(p;R)}w,\ \ \max_{\partial \mc B(p;2\theta R)}w\leq e^{C_1}\min_{\partial\mc B(p;2\theta R)}w.\]
Hence the desired inequality follows.
\end{proof}

\subsection{One-sided Harnack inequalities}
In this section, $M^{n+1}$ is always a closed manifold and $\mc N$ is an embedded compact minimal hypersurface in $M$ so that $\mc B(p;1)\cap \partial\mc N=\emptyset$.
\begin{theorem}\label{thm:one sided harnark for high dimension}
Let $4\leq (n+1)\leq 7$. Given $C_1,K>0$, $I\in \mathbb N$, there exist $C=C(M,\mc N,C_1,K,I)$, $C_0=C_0(M,\mc N)$ and $R_0=R_0(M,\mc N,C_1,K)$ so that if $\Gamma,\Sigma$ are minimal graphs with functions $v,u$ on $\mc A(p;r,2R)\setminus\bigcup_{j=1}^I\mc B(q_j;r)$ for $\{q_j\}\subset \mc A(p;4r,R/2)$, $R_0/4\geq R\geq C_0I r$ and $(v,u)$ is a strong $(C_1|x|^2,K)$-pair (see Definition \ref{def:fK pair}), where $|x|=\dist_\mc N(x,p)$, then there exists $\wti r\leq C_0Ir$ so that
\begin{equation}
\max_{x\in\partial \mc B(p;R)}(v-u)(x)\leq C\min_{x\in\partial\mc B(p;\wti r)}(v-u)(x).
\end{equation}
As a corollary,
\begin{equation}\label{eq:high dim thm:max value}
	\max_{\partial \mc B(p;R)}(v-u)(x)\leq Cr^2.
\end{equation}
\end{theorem} 

\begin{theorem}\label{thm:harnack for minimal graphs on minimal surfaces}
Let $n=3$. Given $C_1,K>0$, $I\in\mb N$, there exist $C=C(M,\mc N,C_1,K,I)$, $\epsilon_0=\epsilon_0(M,\mc N,C_1,K,I)$ so that if $\Gamma,\Sigma$ are minimal graphs with functions $v,u$ on $\Xi:=\mc A(p;r^2,2\epsilon)\setminus\bigcup_{j=1}^I\mc B(q_j;r^2)$ for some $\epsilon<\epsilon_0$ and 
\begin{itemize}
\item $(v,u)$ is a strong $(C_1(|x|^2+r^2),K)$-pair (see Definition \ref{def:fK pair}), where $|x|=\dist_\mc N(x,p)$;
\item $\{q_j\}\subset \mc B(p;s)$ for some $s\in[\sqrt \epsilon r/2,\epsilon/2]$;
\end{itemize}
then we have
\begin{equation*}
\max_{\partial \mc B(p;\epsilon)}(v-u)(x)\leq C\min_{\partial\mc B(p;s)}(v-u)(x).
\end{equation*}
\end{theorem}

\begin{remark}
Theorem \ref{thm:harnack for minimal graphs on minimal surfaces} is equivalent to that statement for $s=\sqrt{\epsilon}r$. Namely, let $R=s/\sqrt\epsilon$. Then $R\geq r$ and $v,u$ are minimal graph functions over $\mc A(p;R^2,2\epsilon)\setminus\bigcup_{j=1}^I\mc B(q_j;R^2)$ and $\{q_j\}\subset\mc B(p;\sqrt{\epsilon }R)$.
\end{remark}

\section{Existence of Jacobi fields}\label{sec:to free boundary minimal hypersurface}
Let $(M^{n+1},\partial M,g)$ be a compact manifold with boundary of dimension $3\leq (n+1)\leq 7$. Recall that $\mc M(\Lambda,I)$ is the space of almost properly embedded free boundary minimal hypersurfaces with $\mathrm{index}\leq I$ and $\area\leq \Lambda$.

We first recall the following compactness theorem:
\begin{theorem}[\cite{GZ18}*{Theorem 4.1}]\label{thm:compactness thm}
Let $\{\Sigma_k\}\subset \mc M(\Lambda,I)$. Then up to a subsequence, $\Sigma_k$ converges smoothly and locally uniformly to $\Sigma$ on $\Sigma\setminus \mc{W}$ with finite multiplicity, where $\mc W\subset \Sigma$ is a finite subset. Moreover, if the convergence has multiplicity one and $\Sigma_k\neq\Sigma$ eventually, then $\Sigma$ has a non-trivial Jacobi field.
\end{theorem}

We now review the convergence. We assume that $\Sigma$ is two-sided.

Let $\mf{n}$ be the unit normal of $\Sigma$ and $X\in \mathfrak{X}(M,\Sigma)$ (see \cite{GZ18}*{\S 2}) be an extension of $\mf{n}$. Suppose that $\phi_t$ is a one-parameter family of diffeomorphisms of $\wti M$ generated by $X$. For any domain $U\subset \Sigma$ and  small $\delta>0$, $\phi_t$ produces a neighborhood $U_\delta$ of $U$ with thickness $\delta$, i.e., $U_\delta=\{\phi_t(x)\,|\, x\in U, |t|\leq \delta\}$. If $U$ is in the interior of $\Sigma$, then $U_\delta$ is the same as $U\times [-\delta,\delta]$ in the geodesic normal coordinates of $\Sigma$ for $\delta$ small.  Now fix a domain $\Omega\subset \subset \Sigma\setminus \mc{W}$, by the convergence $\Sigma_k\to \Sigma$, we know that for $k$ sufficiently large, $\Sigma_k\cap \Omega_\delta$ can be decomposed as $m$ graphs over $\Omega$ which can be ordered by height
\[u_k^1<u_k^2<\cdots<u_k^m.\]

Theorem \ref{thm:compactness thm} says that $\Sigma$ is degenerate when $m=1$. In this paper, we improve Theorem \ref{thm:compactness thm}:
\begin{theorem}\label{thm:degenerate theorem}
Let $\{\Sigma_k\}\subset \mc M(\Lambda,I)$ as in Theorem \ref{thm:compactness thm}. Suppose that $m\geq 2$. Then $\Sigma$ is degenerate, i.e. $\Sigma$ has a non-trivial Jacobi field.
\end{theorem}
\begin{proof}
If $\Sigma$ is one-sided, we can then construct a non-trivial Jacobi field over $\wti{\Sigma}$ and the construction is similar to the case when $\Sigma$ is two-sided. Hence, in the following, we will assume that $\Sigma$ is two-sided.

For any $p\in\mc W$ and $\epsilon\ll 1$, set 
\[\lambda_k(p,\epsilon)=\max_{\partial \mc B(p;\epsilon )}\{u^m_k,-u^1_k\}\ \ \text{ and }\ \ \Lambda_{k,\epsilon}=\max_{p\in\mc W}\lambda_k(p,\epsilon).\]

Set $w_k=u_k^m-u_k^1$ and $\overline w_k=w_k/\Lambda_{k,\epsilon}$. Taking an exhaustion $\{\Omega_i\}$ of $\Sigma\setminus\mc W$, we obtain a Jacobi field $w$ on $\Sigma\setminus \mc W$. Note that $w$ may be trivial or unbounded.

We pause to give the following claim, which is from \cite{GZ18}*{Claim D}.
\begin{claim}\label{claim:one is good}
For each $p\in\mc W$, there exists a constant $C=C(M,\Sigma,p,\epsilon)$ such that either $j=1$ or $j=m$ satisfies the following
\[\limsup_{r\rightarrow 0}\limsup_{k\rightarrow\infty}\frac{\max_{\partial \mc B(p;r)}|u_k^j|}{\max_{\partial\mc B(p;\epsilon)}u_k^j}\leq C.\]
\end{claim}
\begin{proof}[Proof of Claim \ref{claim:one is good}]
First assume that $p\in\mathrm{Int \Sigma}\cap\partial M$. Then without loss of generality, we assume that $\partial M$ lies on the negative side of $\Sigma$ near $p$ as in \cite{GZ18}. Then \cite{GZ18}*{Page 19, Claim C} gives that for any $r\in (0,\epsilon)$, 
\[\max_{x\in\partial \mc B(p;r)}u^m_k>0 \text{\ \  for $k$ sufficiently large.}\]
Then the conclusion of Claim \ref{claim:one is good} for $j=m$ follows from the argument in \cite{GZ18}*{Page 20-21, Claim D}.

It remains to consider $p\notin\mathrm{Int}\Sigma\cap\partial M$. Then the minimal foliation argument works for both $u_k^1$ and $u^m_k$. Note that either
\[\max_{x\in\partial \mc B(p;r)}u^m_k>0 \text{\ \  for $k$ sufficiently large,}\]
or
\[\max_{x\in\partial \mc B(p;r)}-u^1_k>0 \text{\ \  for $k$ sufficiently large,}\]
Then the desired result also follows from the argument in \cite{GZ18}*{Page 20-21, Claim D}.
\end{proof}

We first consider the case $w=0$ on $\Sigma$. Then we set $\overline u_k^j=u^j_k/\Lambda_{k,\epsilon}$ for $1\leq j\leq m$. Then for any $\Omega\subset\subset \Sigma\setminus\mc W$, $\overline u^m_k$ is uniformly bounded. Hence $\overline u^m_k$ locally smoothly converges to a Jacobi field $\overline u$. It follows that $\overline u^1_k\rightarrow \overline u$ since $w=0$. 

Therefore, $\overline u$ is smooth through $\mc W$. Then by the definition of $\Lambda_{k,\epsilon}$, we have either $\max \overline u^m_k=1$ or $\max(-\overline u^1_k)=1$. This gives that $\overline u$ is non-trivial. Thus we also get a nontrivial Jacobi field in this case.

It remains to consider $w$ is non-trivial. Then Theorem \ref{thm:degenerate theorem} follows from this lemma:
\begin{lemma}\label{lem:w is bounded}
$w$ is bounded.
\end{lemma}
We postpone the proof to the next subsection.

Note that for $\Omega\subset\subset\Sigma\setminus \mc W$, $w$ is uniformly bounded.
Denote by $\mc W_0$ the subset of $\mc W\cap\partial M\setminus \partial \Sigma$ so that 
\[\text{ given } \epsilon>0, B^M(p;\epsilon)\cap \partial \Sigma_k\neq \emptyset \text{ for $k$ sufficiently large}.\]

It follows from \cite{ACS17}*{Section 6} (see also \cite{GZ18}*{Section 2.3}) that $w$ is smooth through $W\setminus \mc W_0$.
\end{proof}

\begin{remark}
Note that in the Proof of Theorem \ref{thm:degenerate theorem}, for $p\notin\mc W$, $w_k/w_k(p)$ always converges to a positive Jacobi field $w'$ with possibly discrete singularities on $\mc W$, where $w'$ may be infinity by Lemma \ref{lem:gradient estimates}. Then a classical PDE theory (a cut-off trick) shows that $\Sigma$ is stable, which implies that the Jacobi field is positive.
\end{remark}

The following subsections are devoted to the proof lemma \ref{lem:w is bounded}.
\subsection{Proof of Lemma \ref{lem:w is bounded}}
We prove it by a contradiction argument. Suppose that $w$ is unbounded.

Without loss of generality, we assume that $\partial M$ lies on the negative side of $\Sigma$ around $p$. Then by the Claim D in \cite{GZ18}, there exists a constant $C=C(K,\epsilon)$ such that
\[\limsup_{r\rightarrow 0}\limsup_{k\rightarrow\infty}\frac{\max_{\partial \mc B(p;r)}|u_k^m|}{\max_{\partial \mc B(p;\epsilon)}u_k^m}\leq C.\]
Hence $u^m_k/\Lambda_{k,\epsilon}$ is uniformly bounded near $p$. Together with the assumption of $w$ is unbounded, then we have
$u^1_k/\Lambda_{k,\epsilon}$ is unbounded around $p$ as $k\rightarrow\infty$. Then for any $\kappa>0$ (would be fixed later), we can shrink $\epsilon$ so that for $k$ sufficiently large, 
\begin{equation}\label{eq:kappa}
\max_{\partial\mc B(p;\epsilon) }{-u_k^1}>\kappa \cdot \max_{\partial\mc B(p;\epsilon)}u^m_k.
\end{equation}

Recall that $\partial M$ is smooth. Hence there exists a constant $C_1>1$ so that the graph function $u_{\partial M}$ of $\partial M$ on $B(p;\epsilon)$ satisfying
\begin{equation}\label{eq:second order boundary}
u_{\partial M}\geq -C_1|x|^2 \ \ \text{ for } |x|\leq\epsilon,\ \ \text{ where } |x|=\dist_\Sigma(x,p).
\end{equation}
We can also take $\delta$ small enough so that the minimal foliation near $\mc B(p;\epsilon)$ containing $\mc B(p;\epsilon)\times [-\delta,\delta]$. Denote by $\pi$ the projection to $\Sigma$.

Set $h=\min_{\partial\mc B(p;\epsilon)}-u_k^1$ and
\[S_k=\{x\in\Sigma_k\cap \mc B(p;\epsilon)\times[-\delta,\delta]:|\langle\n_k,\nabla d\rangle|\leq 1/2\},\]
where $\n_k$ is the unit normal vector field of $\Sigma_k$ and $d$ is the signed distance function to $\Sigma$. Then $S_k$ is a closed set of $\Sigma_k$. Note that $\epsilon$ can be taken small enough so that $\partial \Sigma_k\cap (B(p;\epsilon)\times [-\delta,\delta])\subset S_k$.

Let $\rho(x)=L(|\pi(x)|^2+h/\kappa)$, where $L$ is a constant (to be specified later).
\begin{claim}\label{claim:I balls cover Sk}
There exist $\{x_j\}_{j=1}^I\subset \Sigma_k$ so that 
\[S_k\subset \bigcup_{j=1}^I B^{M}(x_j;\rho(x_j)).\] 
\end{claim}
\begin{proof}[Proof of Claim \ref{claim:I balls cover Sk}]
First take any $x_1\in S_k$ so that 
\[\rho(x_1)=\max_{x\in S_k}\rho(x).\]
If we have $x_1,...,x_j$, then take $x_{j+1}\in S_k\setminus \bigcup_{l=1}^jB^M(x_l;\rho(x_l))$ so that
\[\rho(x_{j+1})=\max\{\rho(x):x\in S_k\setminus \bigcup_{l=1}^jB^M(x_l;\rho(x_l))\}.\]
If the process does not stop in $I$ steps, then there exists $\{x_j\}_{j=1}^{I+1}\subset \Sigma_k$ satisfying
\[\dist_M(x_j,x_i)\geq \rho(x_i), \text{ \ \ for \ \ } j>i.\]
Note that for $j>i$, by the choice of $x_j$, $\rho(x_j)\leq \rho(x_i)$. It follows that  
\[ B^M(x_j;\rho(x_j)/3)\cap B^M(x_i;\rho(x_i)/3)=\emptyset \text{ for } i\neq j.\]
Then applying \cite{GZ18}*{Lemma 2.11}, there exists $y\in \{x_j\}_{j=1}^{I+1}$ so that $\Sigma_k$ is stable in $B^M(y;\rho(y)/3)$ since the Morse index of $\Sigma_k$ is bounded by $I$. Using the curvature estimate \cite{GZ18}*{Theorem 3.2}, we have
\begin{equation}
\label{E:estimates of A_k}
\sup_{x\in\Sigma_k \cap B^M(y;\rho(y)/4)} |A^{\Sigma_k}|^2(x)\leq C_2/(\rho(y))^2
\end{equation}
for some uniform constant $C_2>0$.

Since $y\in S_k$, then there exists $\nu\in T_y\Sigma_k$ so that $\langle\nu,\nabla d\rangle>1/2$. let $\gamma$ be the geodesic starting at $y$ with direction $\nu$. By direct computation,
\begin{align*}
&\frac{d}{ds}\langle \gamma'(s),\nabla d\rangle\\
=&A_k(\gamma'(s),\gamma'(s))\langle\nabla d,\mathbf n_k\rangle+\nabla^2d(\gamma'(s),\gamma'(s))\rangle\\
  =&A_k(\gamma'(s),\gamma'(s))\langle\nabla d,\mathbf n_k\rangle+\nabla^2d((\gamma'(s))^\top,(\gamma'(s))^\top),
\end{align*}
where $A_k$ is the second fundamental form of $\Sigma_k$ and $(\gamma'(s))^\perp$ is the projection to $\{d^{-1}(d(\gamma_k(s)))\}$. Hence for $t\in (0,\rho(y)/(10C_2))$,
\[
\frac{d}{dt} d(\gamma(t))=  \langle \gamma'(t),\nabla d\rangle 
  \geq \frac{1}{2}-\int_0^t (|A_k(\gamma(s))|+1)ds
  \geq \frac{1}{2}-2C_2t/\rho(y)\geq \frac{1}{4}.
\]
Therefore, $\gamma(t)\notin \partial\Sigma_k$ for $t\in (0,\rho(y)/(10C_2))$.

Furthermore, 
\begin{align*}
  d(\gamma(\frac{\rho(y)}{10C_2}))=&d(\gamma(0))+\int_0^{\rho(y)/(10C_2)} \langle\nabla d,\gamma'(t)\rangle\, dt\\
  \geq& -C_1|\pi(y)|^2+\frac{1}{4}\cdot \rho(y)/(10C_2)\\
  \geq& (\frac{1}{40C_2}-\frac{C_1}{L})\rho(y).
\end{align*}
Then we can take $L=L(C_1,C_2)$ large enough so that 
\[(\frac{1}{40C_2}-\frac{C_1}{L})\rho(y)\geq \rho(y)/\sqrt L=\sqrt L\cdot h/\kappa,\]
which leads to a contradiction to our assumptions. This completes the proof of Claim \ref{claim:I balls cover Sk}.
\end{proof}

Now let $\Sigma_k'$ be the component of $\Sigma_k\setminus \bigcup_{j=1}^IB(x_j;\rho(x_j))$ containing the bottom sheet graph on $\partial\mc B(p;\epsilon)$. By the definition of $S_k$, we have $|\langle\n_k,\nabla d\rangle|>1/2$. Thus we conclude that $\Sigma_k'$ is a minimal graph on $\pi(\Sigma_k')$. Denote by $u_k$ the minimal graph function. A standard computation (see Appendix \ref{sec:appendix:minimal graph function}) gives that 
\[|\langle\n_k,\nabla d\rangle|=1/\sqrt{1+|\nabla u_k|^2}.\]
It follows that
\begin{equation}\label{eq:uk gradient bound}
|\nabla u_k(x)|\leq 1, \text{ \ \ for \ \ } x\in\pi(\Sigma_k').
\end{equation}

Note that $u_k$ may not be negative everywhere. To overcome this, we recall the minimal foliation near $\Sigma$.  Let $t=\max_{\partial \mc B(p;\epsilon)}u_k^m$ and $\Sigma_t$ be the slice in the minimal foliation, i.e. $\Sigma_t$ is a minimal graph on $\mc B(p;\epsilon)$ and $v_t=t$ on $\partial \mc B(p;\epsilon)$, where $v_t$ is the graph function. Then for $x\in \mc B(p;\epsilon/2)$,
\begin{equation}\label{eq:inequality for vt}
|\nabla^2 v_t|+|\nabla v_t|\leq K|v_t|,
\end{equation}
for some universal constant $K$. Moreover, by the assumption (\ref{eq:kappa}),
\begin{equation}\label{eq:upper bound inequality for vt}
h\geq \kappa v_t.
\end{equation}
Without loss of generality, we can assume that
\begin{equation}\label{eq:vt gradient bound}
|\nabla v_t(x)|\leq 1 \text{ \ \ for all \ \ } x\in\mc B(p;\epsilon).
\end{equation}

\begin{claim}\label{claim:second order graph on Sigmat}
For $x\in\Sigma$ with $|x|\geq\sqrt{h/\kappa}$,
\[(v_t-u_k)(x)\leq (C_1+1)|x|^2.\]
\end{claim}
\begin{proof}[Proof of Claim \ref{claim:second order graph on Sigmat}]
Note that $v_t\leq h/\kappa\leq |x|^2$ for $|x|\geq \sqrt{k/\kappa}$. Together with (\ref{eq:second order boundary}), we have
\[
(v_t-u_k)(x)\leq |x|^2-u_{\partial M}(x)\leq (C_1+1)|x|^2,
\]
which is the desired inequality.
\end{proof}

Denote by $Q=\{\pi(x_j)\}$ and $s_1=\max_{1\leq j\leq I}\rho(x_j)$. Recall that $\pi$ is the projection to $\Sigma$.

\begin{claim}\label{claim:suitable balls and radius}
There exist $r_1\in [8s_1,4^{3I+3}s_1]$ and $Q_1\subset Q$ satisfying the following:
\begin{enumerate}[(i)]
\item\label{item:new subset:containing} for any $x\in\{x_j\}_{j=1}^I$, there exists $x'\in Q_1$ so that $\pi ( B(x;\rho(x)))\subset\mc B(x'; r_1/4)$ or $\pi(B(x;\rho(x)))\subset \mc B(p; r_1/4)$;
\item\label{item:new subset:disjoint each other} $\mc B(x';4 r_1)\cap \mc B(x'';4 r_1)=\emptyset$ for two different points $x',x''\in Q_1$;
\item\label{item:new subset:disjoint p} $\mc B(x';4 r_1)\cap\mc B(p;4 r_1)=\emptyset$ for all $x'\in Q_1$.
\end{enumerate}
\end{claim}
\begin{proof}[Proof of the Claim \ref{claim:suitable balls and radius}]
Note that for $k\rightarrow\infty$, $\max_{1\leq j\leq I}\dist_M(x_j,\Sigma)\rightarrow 0$. So without loss generality, we can assume that for any $r<1$
\[\pi (B(x_j;r))\subset\mc B(\pi(x_j); 2r).\]
Now let $Q_1= Q$ and $r_1=8s_1$. Then the first item follows immediately. If such $Q_1$ and $r_1$ satisfy all the requirements, then we are done. Otherwise, there exists $y\in Q_1$ so that either
\[\mc B(x';4r_1)\cap \mc B(y;4r_1)\neq \emptyset , \mc B(x';4r_1)\setminus\mc B(y;4r_1)\neq\emptyset \text{\ \ for some } x'\in Q_1, \]
or 
\[\mc B(y;4 r_1)\cap\mc B(p';4 r_1)\neq \emptyset.\]
In both cases, we replace $(Q_1,r_1)$ by $(Q_1\setminus\{y\},64r_1)$. 
Then $\mc B(y;2s_1)\subset\mc B(x';r_1/4)$. Hence (\ref{item:new subset:containing}) still holds true for our new $Q_1$ and $r_1$. 

As far, we have proved that if $Q_1$ and $r_1$ satisfy (\ref{item:new subset:containing}) but not the last two requirements, then we can replace $(Q_1,r_1)$ by $(Q_1\setminus\{y\},64r_1)$ for some $y\in Q_1$ so that the new $Q_1$ and $r_1$ also satisfy (\ref{item:new subset:containing}).

Note that each time we get the new $Q_1$ with fewer element. Thus, such a process  will stop in $N(\leq I)$ steps. Then those our desired $Q_1$ and $r_1$.
\end{proof}

Note that $v_t,u_k$ are positive minimal graph functions on $\Xi:=\mc A(p;r_1,\epsilon)\setminus \cup_{y\in Q_1}\mc B(y;r_1)$. Moreover,
\begin{claim}\label{claim:it is good pair}
If $r_1\geq \sqrt{h/\kappa}$, then $(v_t,u_k)$ is a $((C_1+1)|x|^2,2)$-pair (see Definition \ref{def:fK pair}) on $\Xi$.
\end{claim}
\begin{proof}[Proof of Claim \ref{claim:it is good pair}]
Recall that $t=\max_{\partial \mc B(p;\epsilon)}u^m_k$ and $u_k$ is the graph function of $\Sigma_k'$. Then it follows that $v_t-u_k\geq 0$ on $\Xi$. If $r_1\geq \sqrt{h/\kappa}$, using Claim \ref{claim:second order graph on Sigmat}, then we have
\[v_t-u_k\leq (1+C_1)|x|^2, \text{\ \ for all \ } x\in\Xi.\]
Also, (\ref{eq:uk gradient bound}) and (\ref{eq:vt gradient bound}) gives that 
\[|\nabla u_k(x)|+|\nabla v_t(x)|\leq 2 \text{\ \ for all \ } x\in\Xi.\]
Therefore, $(v_t,u_k)$ is a $((C_1+1)|x|^2,2)$-pair.
\end{proof}

\vspace{1em}
To proceed the proof of Lemma \ref{lem:w is bounded}, we divide it into high-dimensional cases and three-dimensional cases.

{\bf Part I:} In this part, we address the high dimensional case: $4\leq (n+1)\leq 7$.

Let $R_0=R_0(M,\Sigma)$ and $C_0=C_0(M,\Sigma)$ be the constants in Theorem \ref{thm:one sided harnark for high dimension}. Then we can take $\epsilon$ small enough so that $\epsilon<R_0/8$. Then for $k$ large enough so that $r_1\leq 4^{3I+3}s_1\leq \epsilon/(4^{I+3}C_0)$. Now applying Theorem \ref{thm:one sided harnark for high dimension}, there exists $\wti r_1\leq C_04^{I+1}r_1$ so that 
\[\max_{\partial \mc B(p;\epsilon)}(v_t-u_{k})(x)\leq C\min_{\partial\mc  B(p;\wti r_1)}(v_t-u_k)(x),\]
and
\[\mc B(y;r_1)\cap \mc A(p;\wti r_1/2,2\wti r_1)=\emptyset, \text{ \ \ for all \ } y\in Q_1.\]

Now we construct $Q_j\subset Q$ inductively:
\begin{claim}\label{claim:construction of Qr}
Suppose that, for some $j\in\mb N$, $Q_j\subset Q$ and $\wti r_j>C_0\cdot 4^{10I+10}\sqrt{h/\kappa}$ satisfying
\begin{itemize}
\item $\mc B(y;r_j)\cap \mc A(p;\wti r_j/2,2\wti r_j)=\emptyset$, for all $y\in Q_j$;
\item $Q_{j}\neq \emptyset$ and $Q\cap \mc B(p;\wti r_{j})\neq \emptyset$.
\end{itemize} 
Then 
\begin{equation}\label{eq:decrease Qj}
\#(Q\cap \mc B(p;\wti r_{j}))\leq I-j;
\end{equation}
and  there exist a non-empty set $Q_{j+1}\subset Q\cap \mc B(p;\wti r_j)$ and $r_{j+1}<\wti r_j$ so that $(v_t,u_k)$ is a strong $((C_1+1)|x|^2,K)$-pair on $\mc A(p;r_{j+1},\wti r_j)\setminus\cup_{y\in Q_{j+1}}\mc B(y;r_{j+1})$. 
\end{claim}
\begin{proof}
Once we have $\wti r_j$ and $Q_j$, then set \[s_{j+1}=\max\,(\,\max_{x\in Q\cap\mc B(p;\wti r_j)}\rho(x_j),\sqrt{h/\kappa}).\]
By the same process with Claim \ref{claim:suitable balls and radius}, we can take $r_{j+1}\in [8s_{j+1},4^{3I+3}s_{j+1}]$ and $Q_{j+1}\subset Q\cap \mc B(p;\wti r_j)$ satisfying the following:
\begin{itemize}
	\item for any $x\in Q\cap\mc B(p;\wti r_j)$, there exists $x'\in Q_{j+1}$ so that $\pi(B(x;\rho(x)))\subset \mc B(x';r_{j+1}/4)$ or $\pi(B(x;\rho(x)))\subset \mc B(p;r_{j+1}/4)$;
	\item $\mc B(x';4r_{j+1})\cap\mc B(x'';4r_{j+1})=\emptyset$ for two different points $x',x''\in Q_{j+1}$;
	\item $\mc B(x';4r_{j+1})\cap \mc B(p;4r_{j+1})=\emptyset$ for all $x'\in Q_{j+1}$.
\end{itemize}
We now check such $Q_{j+1}$ satisfies our requirements. Recall that $u_k$ is well-defined on 
\[\mc B(p;\epsilon)\setminus\bigcup_{y\in Q}\pi(B(y;\rho(y))).\]
Note that for any $x\in Q\cap\mc B(p;\wti r_j)$, there exists $x'\in Q_{j+1}$ so that $\pi(B(x;\rho(x)))\subset \mc B(x';r_{j+1}/4)$ or $\pi(B(x;\rho(x)))\subset \mc B(p;r_{j+1}/4)$. Together with $r_{j+1}\geq s_{j+1}\geq \sqrt{h/\kappa}$, by Claim \ref{claim:it is good pair}, $(v_t,u_k)$ is a $((C_1+1)|x|^2,2)$-pair on $\mc A(p;r_{j+1},\wti r_j)\setminus\bigcup_{y\in Q_{j+1}}\mc B(y;r_{j+1})$. Together with (\ref{eq:inequality for vt}) and (\ref{eq:upper bound inequality for vt}), we conclude that $(v_t,u_k)$ is a strong $((C_1+1)|x|^2,K)$-pair.

It remains to prove (\ref{eq:decrease Qj}). By the definition of $\wti r_{j}$, there exist $z\in Q_{j-1}$ ($Q_0:=Q$) so that  
\begin{equation}\label{eq:bound of wti r}
\wti r_{j}\leq C_0\cdot 4^{4I+4}\cdot L(|\pi (z)|^2+h/\kappa).
\end{equation}
Recall that $\wti r_j>C_0\cdot 4^{10I+10}\sqrt{h/\kappa}$. This deduces that
\[L(|\pi(z)|^2+h/\kappa)\geq \sqrt{h/\kappa}.\]
Note that $L=L(C_1,C_2)$ depends only on $M$ and $\Sigma$ (see Claim \ref{claim:I balls cover Sk}. Thus, we can take $\epsilon$ small enough so that 
\[\sqrt{h/\kappa}\geq 2L\cdot h/\kappa,\]
which implies that 
\[h/\kappa\leq |\pi(z)|^2.\]
Then (\ref{eq:bound of wti r}) becomes
\[\wti r_j\leq 2C_0L\cdot 4^{4I+4}|\pi(z)|^2\leq |\pi(z)|.\]
We conclude that 
\[\#(Q\cap\mc B(p;\wti r_j))\leq \#(Q\cap \mc B(p;\wti r_{j-1}))-1, \text{  where }\wti r_0=\epsilon.\]
By induction, (\ref{eq:decrease Qj}) follows.
\end{proof}

\vspace{1em}
By Claim \ref{claim:construction of Qr}, $(v_t,u_k)$ is a $((C_1+1)|x|^2,2)$-pair on $\mc A(p;r_{j+1},\wti r_j)\setminus\cup_{y\in Q_{j+1}}\mc B(y;r_{j+1})$. Then applying Theorem \ref{thm:one sided harnark for high dimension} again, there exists $\wti r_{j+1}\leq C_04^{I+1}r_{j+1}$ so that 
\begin{equation}\label{eq:far with wti rj}
\mc B(y;r_{j+1})\cap \mc A(p;\wti r_{j+1}/2,2\wti r_{j+1})=\emptyset, \text{ \ \ for all \ } y\in Q_1,
\end{equation}
and
\[\max_{\partial\mc B(p;\wti r_{j})}(v_t-u_{k})(x)\leq C\min_{\partial\mc B(p;\wti r_{j+1})}(v_t-u_{k})(x),\]
which also implies that 
\[\max_{\partial\mc B(p;\epsilon )}(v_t-u_{k})(x)\leq C^{j+1}\min_{\partial \mc B(p;\wti r_{j+1})}(v_t-u_{k})(x).\]

By (\ref{eq:decrease Qj}), such processes must stop in $N(\leq I)$ steps. Using Claim \ref{claim:construction of Qr}, together with (\ref{eq:far with wti rj}) and $Q_{j}\neq \emptyset$ for all $j\leq N$, we conclude that
\begin{itemize}
\item \underline{either} $\wti r_{N}\leq C_0\cdot 4^{10I+10}\sqrt{h/\kappa}$;
\item \underline{or} $Q\cap\mc B(p;\wti r_N)=\emptyset$.
\end{itemize}

In the first case, then for $x\in \mc B(p;\wti r_{N})$
\[(v_t-u_k)(x)\leq 2C_1\wti r^2_N\leq Ch/\kappa, \]
which implies
\begin{equation}\label{eq:final inequality}
\max_{\partial\mc B(p;\epsilon )}(v_t-u_{k})(x)\leq C^{N}\min_{\partial \mc B(p;\wti r_{N})}(v_t-u_{k})(x)\leq Ch/\kappa.
\end{equation}

In the second case, $(v_t,u_k)$ is a $((C_1+1)|x|^2,2)$-pair on $\mc A(p;\sqrt{h/\kappa},2\wti r_N)$, and hence a strong $(((C_1+1)|x|^2,K)$-pair by (\ref{eq:inequality for vt}) and (\ref{eq:upper bound inequality for vt}). Using Theorem \ref{thm:one sided harnark for high dimension} again,
\[\max_{x\in\partial\mc B(p;\wti r_N)}(v_t-u_k)(x)\leq\min_{x\in\partial\mc B(p;\sqrt{h/\kappa})}(v_t-u_k)(x)\leq  C\cdot C_1h/\kappa, \]
which also implies (\ref{eq:final inequality}).

Recall that the constant $C$ in (\ref{eq:final inequality}) depends only on $M,\Sigma, I,C_1$. Thus, we can take $\kappa$ larger than such $C$, then (\ref{eq:final inequality}) contradicts
\[h\leq \max_{\partial\mc B(p;\epsilon )}(v_t-u_{k})(x).\]

Therefore, we complete the proof for $4\leq (n+1)\leq 7$.

\vspace{2em}
{\bf Part II:} In this part, we address the three-dimensional case.

Let $Q_1$ and $r_1$ be the notion in Claim \ref{claim:suitable balls and radius}. Assume that $Q_1=\{y_j\}$. Recall that $s_1=\max_{y\in Q}\rho(y)$.

\begin{claim}\label{claim:suitable r1}
Suppose that $Q_1\neq \emptyset$ and $s_1\geq 2Lh/\kappa$. There exist $t_1,\wti t_1$ satisfy 
\begin{itemize}
\item $2t_1\leq \max_{y\in Q_1}|y|<\wti t_1/2$ ;
\item $\wti t_1\leq 2^{10I+1}t_1$ and $\wti t_1\geq \sqrt{\epsilon\cdot r_1}$;
\item $\mc B(y;r_1)\subset \mc B(p;\wti t_1)$ for all $y\in Q_1$.
\end{itemize}  
\end{claim}
\begin{proof}[Proof of Claim \ref{claim:suitable r1}]
Let $\alpha_j=\log_2(|y_j|/r_1)$. Then by the Lemma \ref{lem:number lemma}, we can find $2N(\leq 2I)$ non-negative intergers $\{k_j\}_{j=1}^{2N}$ such that 
\begin{itemize}
	\item $k_{j+1}-k_{j}\geq 3$;
	\item $k_{2j}-k_{2j-1}\leq 10I+1$.
	\item $\{\alpha_j\}_{j=1}^I\subset \cup_{j}[k_{2j-1}+1,k_{2j}-1]$.
\end{itemize}
The last item is equivalent to say
\[Q_1\subset \bigcup_{j}\mc A\mc (p;2^{k_{2j-1}+1}r_1,2^{k_{2j}-1}r_1).\]
which implies that 
\[\bigcup_{y\in Q_1}\mc B(y;r_1)\subset \mc B(p;2^{k_{2N}}r_1).\] 
Now we set $t_1:=2^{k_{2N-1}}r_1$ and $\wti t_1:=2^{k_{2N}}r_1$. Then it follows that such $t_1$ and $\wti t_1$ satisfy the first and the third items.

It remains to prove that $\wti r_1\geq \sqrt{\epsilon\cdot r_1}$. Note that $s_1\geq 2Lh/\kappa$. Then there exists $z\in Q$ so that 
\[\rho(z)=L(|z|^2+h/\kappa)\geq 2Lh/\kappa,\]
which deduces that $|z|\geq \sqrt{h/\kappa}$. Recall that $r_1\leq 4^{3I+3}s_1$. Thus, we have
\[
\sqrt{\epsilon\cdot r_1}\leq \sqrt{\epsilon\cdot s_1}\cdot 2^{3I+3}\leq \sqrt{2\epsilon L|z|^2}\cdot 2^{3I+3}\leq  |z|\leq 2^{k_{2N}}r_1=\wti t_1.
\]
Thus, we have proved Claim \ref{claim:suitable r1}.
\end{proof}

Now we can construct $s_j,Q_j,t_j,\wti t_j$ inductively. Suppose we have $t_j$, then set
\[s_{j+1}=\max_{x\in Q\cap \mc B(p;t_j)}\rho(x).\]
By the same process with Claim \ref{claim:suitable balls and radius}, we can take $r_{j+1}\in[8s_{j+1},4^{3I+3}s_{j+1}]$ and $Q_{j+1}\subset Q\cap \mc B(p;t_j)$ so that 
\begin{itemize}
\item for any $x\in Q\cap \mc B(p;t_j)$, there exists $x'\in Q_{j+1}$ so that $\pi(B(x;\rho(x)))\subset \mc B(x; r_{j+1}/4)$ or $\pi(B(x;\rho(x)))\subset \mc B(p; r_{j+1}/4)$;
\item $\mc B(x';4 r_{j+1})\cap\mc B(x'';4r_{j+1})=\emptyset$ for two different $x',x''\in Q_{j+1}$;
\item $\mc B(x';4 r_{j+1})\cap\mc B(p;4 r_{j+1})=\emptyset$ for all $x'\in Q_{j+1}$.
\end{itemize}

\begin{claim}\label{claim:inductive Qtj}
Suppose that $Q_{j+1}\neq \emptyset$ and $s_{j+1}\geq 2Lh/\kappa$. Then there exist $t_{j+1}$ and $\wti t_{j+1}$ satisfy
\begin{itemize}
\item $2t_{j+1}\leq \max_{y\in Q_{j+1}}|y|<\wti t_{j+1}/2$ ;
\item $\wti t_{j+1}\leq 2^{10I+1}t_{j+1}$ and $\wti t_{j+1}\geq \sqrt{\epsilon\cdot r_{j+1}}$;
\item $\mc B(y;r_{j+1})\subset \mc B(p;\wti t_{j+1})$ for all $y\in Q_{j+1}$.
\end{itemize} 
\end{claim}
\begin{proof}[Proof of Claim \ref{claim:inductive Qtj}]
The proof is almost the same with that of Claim \ref{claim:suitable r1}.

Let $\beta_j=\log_2(|y_j|/r_{j+1})$. Then by the Lemma \ref{lem:number lemma}, we can find $2N(\leq 2I)$ non-negative intergers $\{k_m\}_{m=1}^{2N}$ such that 
\begin{itemize}
	\item $k_{m+1}-k_{m}\geq 3$;
	\item $k_{2m}-k_{2m-1}\leq 10I+1$.
	\item $\{\alpha_m\}_{m=1}^I\subset \cup_{m}[k_{2m-1}+1,k_{2m}-1]$.
\end{itemize}
The last item is equivalent to say
\[Q_{j+1}\subset \bigcup_{m}\mc B(p;2^{k_{2m-1}+1}r_{j+1},2^{k_{2m}-1}r_{j+1}).\]
which implies that 
\[\bigcup_{y\in Q_{j+1}}\mc B(y;r_{j+1})\subset \mc B(p;2^{k_{2N}}r_{j+1}).\] 
Now we set $t_{j+1}:=2^{k_{2N-1}}r_{j+1}$ and $\wti t_{j+1}:=2^{k_{2N}}r_{j+1}$. Then it follows that such $t_{j+1}$ and $\wti t_{j+1}$ satisfy the first and the third items.

It remains to prove that $\wti r_{j+1}\geq \sqrt{\epsilon\cdot r_{j+1}}$. Note that $s_{j+1}\geq 2Lh/\kappa$. Then there exists $z\in Q\cap \mc B(p;t_j)$ so that 
\[\rho(z)=L(|z|^2+h/\kappa)\geq 2Lh/\kappa,\]
which deduces that $|z|\geq \sqrt{h/\kappa}$. Recall that $r_{j+1}\leq 4^{3I+3}s_{j+1}$. Thus, we have
\[
\sqrt{\epsilon\cdot r_{j+1}}\leq \sqrt{\epsilon\cdot s_{j+1}}\cdot 2^{3I+3}\leq \sqrt{2\epsilon L|z|^2}\cdot 2^{3I+3}\leq  |z|\leq 2^{k_{2N}}r_{j+1}=\wti t_{j+1}.
\]
Thus, we have proved Claim \ref{claim:inductive Qtj}.
\end{proof}

According to the construction, $Q\cap\mc A(p;t_{j+1},t_{j})\neq \emptyset$. Thus, such inductive process must stop in $N+1(\leq I)$ steps, that is, there $Q_{N+1}=\emptyset$ or $s_{N+1}\leq 2Lh/\kappa$.
\begin{claim}\label{claim:3d good pair}
For $j\leq N$, $(v_t,u_k)$ is a strong $(2C_1(|x|^2+r_{j+1}),K)$-pair on 
\[\mc A(p;r_{j+1},2t_j)\setminus \bigcup_{Q_{j+1}}\mc B(y;r_{j+1}).\]
\end{claim}
\begin{proof}[Proof of Claim \ref{claim:3d good pair}]
For $x\in\mc B(p;\sqrt{h/\kappa})$, we have 
\[(v_t-u_k)(x)\leq h/\kappa+C_1h/\kappa\leq 2C_1r_{j+1};\]
and if $|x|\geq \sqrt{h/\kappa}$, 
\[(v_t-u_k)(x)\leq h/\kappa+C_1|x|^2\leq 2C_1|x|^2.\]
Then together with (\ref{eq:inequality for vt})(\ref{eq:upper bound inequality for vt}) and Claim \ref{claim:it is good pair}, the desired result follows.
\end{proof}

By the definition of $\wti r_{j+1}$, for $y\in Q_{j+1}$,
\[\mc B(y;r_{j+1})\subset\mc B(p;\wti t_{j+1}).\]
Also, from Claim \ref{claim:inductive Qtj}, 
\[\sqrt{t_j\cdot r_{j+1}}\leq\sqrt{\epsilon\cdot r_{j+1}}\leq \wti t_{j+1}.\]
Thus, applying Theorem \ref{thm:harnack for minimal graphs on minimal surfaces}, we have
\[\max_{\partial \mc B(p;t_j)}(v_t-u_{k})(x)\leq C\min_{\partial\mc B(p;\wti t_{j+1})}(v_t-u_{k})(x).\]
Note that $\wti t_{j}\leq 2^{10I+1}t_{j}$. Together with Corollary \ref{cor:harnack for equivalent radius}, we have 
\[\max_{\partial\mc B(p;\wti t_j)}(v_t-u_{k})(x)\leq C\min_{\partial \mc B(p;t_{j})}(v_t-u_{k})(x).\]
From these two inequalities, we conclude that
\begin{equation}\label{eq:from epsilon to N}
\max_{\partial\mc B(p;\epsilon)}(v_t-u_{k})(x)\leq C^{2N}\min_{\partial\mc B(p;t_{N})}(v_t-u_{k})(x).
\end{equation}

Without loss of generality, we assume that $t_{N}\geq \sqrt{h/\kappa}$. Recall that $Q_{N+1}=\emptyset$ or $s_{N+1}\leq 2Lh/\kappa$.

\underline{If $Q_{N+1}=\emptyset$}, then it follows that $Q\cap\mc B(p;t_{N})=\emptyset$ (where $t_0=\epsilon$). In this case, $(v_t,u_k)$ is a strong $(2C_1(|x|^2+h/\kappa),K)$-pair on $\mc A(p;h/\kappa,2t_N)$. Using Theorem \ref{thm:harnack for minimal graphs on minimal surfaces}, we have
\[\max_{\partial\mc B(p;t_N)}(v_t-u_{k})(x)\leq C\min_{\partial\mc B(p;\sqrt{h/\kappa})}(v_t-u_{k})(x)\leq Ch/\kappa.\]
Together with (\ref{eq:from epsilon to N}), we conclude that
\begin{equation}\label{eq:3d final inequality}
\max_{\partial\mc B(p;\epsilon)}(v_t-u_{k})(x)\leq C^{2N+1}h/\kappa.
\end{equation}

\underline{If $s_{N+1}\leq 2Lh/\kappa$}, then for $x\in Q\cap\mc B(p;t_N)$,
\[\rho(x)=L(|x|^2+h/\kappa)\leq 2Lh/\kappa,\]
which implies that $|x|\leq \sqrt{h/\kappa}$. In this case, for $x\in\mc A(p;r_{N+1}+h/\kappa,2t_N)\setminus\cup_{Q_{N+1}}\mc B(y;r_{N+1}+h/\kappa)$,
\[(v_t-u_k)(x)\leq C_1|x|^2+h/\kappa\leq 2C_1(|x|^2+r_{N+1}+h/\kappa).\]
Therefore, $(v_t,u_k)$ is also a $(2C_1(|x|^2+r_{N+1}+h/\kappa),2)$-pair on $\mc A(p;r_{N+1}+h/\kappa,2t_N)\setminus\cup_{Q_{N+1}}\mc B(y;r_{N+1}+h/\kappa)$, and hence a strong $(2C_1(|x|^2+r_{N+1}+h/\kappa),K)$-pair by (\ref{eq:inequality for vt}) and (\ref{eq:upper bound inequality for vt}). Moreover,
\[\sqrt{t_N\cdot(r_{N+1}+h/\kappa)}\leq \sqrt{\epsilon\cdot 4^{3I+3}s_{N+1}}+\sqrt{\epsilon\cdot h/\kappa}\leq 2^{3I+5}L\sqrt{\epsilon\cdot h/\kappa}\leq \frac{1}{2}\sqrt{h/\kappa}.\]
Then Theorem \ref{thm:harnack for minimal graphs on minimal surfaces} gives that
\[\max_{\partial\mc B(p;t_N)}(v_t-u_{k})(x)\leq C\min_{\partial\mc B(p;2\sqrt{h/\kappa})}(v_t-u_{k})(x)\leq Ch/\kappa,\]
together with \ref{eq:from epsilon to N}, which also implies (\ref{eq:3d final inequality}).

Recall that the constant $C$ in (\ref{eq:3d final inequality}) depends only on $M,\Sigma, I,C_1$. Thus, we can take $\kappa$ larger than such $C^{2N+1}$, then (\ref{eq:3d final inequality}) contradicts
\[h\leq \max_{\partial\mc B(p;\epsilon )}(v_t-u_{k})(x).\]

Therefore, we complete the proof for $(n+1)=3$.

\section{High dimensional case}\label{sec:high dim}
In this section, we prove the one-sided Harnack inequality for minimal graph functions in high dimensional cases. Such a result is one of key ingredients in the proof of the existence of Jacobi field in Theorem \ref{thm:degenerate theorem}.

In this section, $(M^{n+1},g)$ is always a closed manifold and $\mc N$ is an embedded compact minimal hypersurface in $M$ so that $\mc B(p;1)\cap \partial\mc N=\emptyset$. Recall that $\mc B(p;r)$ the intrinsic geodesic ball of $\mc N$ with radius $r$ and center $p\in \mc N$ and $\mc A(p;r,s)=\mc B(p;s)\setminus\mc B(p;r)$.

In order to prove Theorem \ref{thm:one sided harnark for high dimension}, we prove the following lemma first:
\begin{lemma}\label{lem:maximum for annulus}
Given $C_1>0$, there exist $C=C(M,\mc N,C_1,K)$ and $R_0=R_0(M,\mc N,C_1,K)$ so that if $\Sigma$ and $\Gamma$ are minimal graphs with functions $v,u$ over $\mc A(p;r,2R)$ for $8r\leq R\leq R_0$  and $(v,u)$ is a strong $(C_1|x|^2,K)$-pair (see Definition \ref{def:fK pair}), where $|x|=\dist_\mc N(x,p)$, then we have
\begin{equation}\label{eq:one-sided harnack for high dim}
\max_{\partial\mc B(p;R)}(v-u)(x)\leq C\min_{\partial\mc B(p;2r)}(v-u)(x).
\end{equation}
\end{lemma}

Now we can use Lemma \ref{lem:maximum for annulus} to prove Theorem \ref{thm:one sided harnark for high dimension} as follows.

\begin{proof}[Proof of Theorem \ref{thm:one sided harnark for high dimension}]
Let $ C_0$ be the constant in Lemma \ref{lem:enlarge interior radius}. Then for any $r_1,r_2\geq C_0Ir$, there exists a $C^1$ curve $\gamma:[0,1]\rightarrow\mc A (p;r_1,r_2)$ satisfying
\begin{itemize}
\item $\gamma(0)\in\partial\mc B(p;r_1)$ and $\gamma(1)\in\partial\mc B(p;r_2)$;
\item $\mathrm{Length}(\gamma)\leq C_0 (r_2-r_1)$;
\item $\dist(\gamma,\cup_{j=1}^I\mc B(q_j,r))\geq r_1/(C_0I)$.
\end{itemize} 

Now take $\wti r\in(C_0Ir,3C_0Ir)$ so that 
\[\dist_\mc N(\partial\mc B(p;\wti r),\bigcup_{j=1}^I\mc B(q_j,r))\geq r.\]

Set 
\[\alpha_j:=\log_2(\dist_\mc N(p,q_j)/\wti r),  \ \ \text{ for all } j\leq I.\]
Then Lemma \ref{lem:number lemma} implies that there exist $2J(\leq 2I)$ non-negative intergers $\{k_j\}_{j=1}^{2J}$ such that 
\begin{itemize}
	\item $k_{j+1}-k_{j}\geq 3$;
	\item $k_{2j}-k_{2j-1}\leq 10I+1$;
	\item $\{\alpha_j\}_{i=1}^{I}\subset \cup_{j}[k_{2j-1}+1,k_{2j}-1]$.
\end{itemize}
The last item deduces that 
\[\bigcup_{j=1}^I\mc B(q_j,r)\subset\cup_{j=1}^{J}\mc B(p;2^{k_{2j-1}+1}\wti r,2^{k_{2j}-1}\wti r).\]
By the choice of $C_0$ in the beginning, Corollary \ref{cor:harnack for equivalent radius} gives that
\begin{equation}\label{eq:high dim:from odd to even}
\max_{\partial\mc B(p;2^{k_{2j}}\wti r)}(v-u)(x)\leq C\min_{\partial \mc B(p;2^{k_{2j-1}}\wti r)}(v-u)(x).
\end{equation}

Notice that there is no $B(q_j;r)$ in $B(p;2^{k_{2j}+1}\wti r,2^{k_{2j+1}-1}\wti r)$. Let $R_0$ satisfy the requirements in Lemma \ref{lem:maximum for annulus}, then 
\begin{equation}\label{eq:high dim:from even to odd}
	\max_{\partial\mc B(p;2^{k_{2j+1}}\wti r)}(v-u)(x)\leq C\min_{\partial\mc B(p;2^{k_{2j}}\wti r)}(v-u)(x).
\end{equation}
Similarly,  
\[\max_{\partial\mc B(p;R)}(v-u)(x)\leq C\min_{\partial\mc B(p;2^{k_{2J}}\wti r)}(v-u)(x).\]
Together with (\ref{eq:high dim:from odd to even}) and (\ref{eq:high dim:from even to odd}, we conclude that
\[\max_{\partial\mc B(p;R)}(v-u)(x)\leq C^{2J+1}\min_{\partial\mc B(p;2^{k_{2N}}\wti r)}(v-u)(x),\]
which is exactly the desired result.
\end{proof}

Then the rest of this section is devoted to prove Lemma \ref{lem:maximum for annulus}.

\subsection{Almost harmonic functions}
In this subsection, we prove an one-sided Harnack for a class of `almost harmonic functions' under additional assumptions.

Let $\mc N^n$ be a compact manifold with boundary of dimension $n\geq 2$. Denote by $\mc B(p;r)$ the geodesic ball of $\mc N$ so that $\mathrm{inj_{\mc N}(p)},\dist_\Sigma(p,\partial\mc N)>r$, where $\mathrm{inj_{\mc N}(p)}$ is the injective radius. We also denote $\mc A(p;r,s)=\mc B(p;s)\setminus\mc B(p;r)$. 

Given a $C^2$ function $w$ on $\partial\mc B(p;R)$, define
\begin{equation}\label{eq:def of Is}
\mc I(s)=s^{1-n}\int_{\partial\mc B(p;s)}w\,d\mu,
\end{equation}
for $s\in (0,R]$.

\begin{lemma}\label{lem:ode}
Let $n\geq 3$. Given $C_2>0$, there exists $R=R(n,C_2)>0$ so that if $\mc I$ is a $C^1$ function on $[r,R]$ with $R\geq4r$ and
\begin{equation}\label{eq:ode inequality}
|(\mc I(s)+cs^{2-n})'|\leq C_2s\mc I(s), \text{ for some constant } c,
\end{equation}
then $\max_{s\in[1,R]}\mc I(s)\leq 6\mc I(2r)$. 
\end{lemma}

\begin{proof}
We divide the proof into two cases.

{\bf Case 1:} $c\leq 0$.

Then $\mc I'\leq C_2s\mc I$, which implies that $\mc I(s)\leq e^{C_2s^2}\mc I(2r)$.
It follows that $\mc I(s)\leq e\mc I(2r)\leq 6\mc I(2r)$ whenever $C_2R^2\leq 1$.

{\bf Case 2: $c>0$.}

Set $J(s)=\mc I(s)+cs^{2-n}$. Then we have $|J'|\leq  C_2sJ$. The argument in Case 1 gives that $J(s)\leq e^{1/10}J(r)\leq 2J(r)$ whenever $C_2R^2\leq 1/10$.

Integrating (\ref{eq:ode inequality}) from $r$ to $2r$, then
\[\mc I(r)+cr^{2-n}-\mc I(2r)-c(2r)^{2-n}\leq  4C_2r^2(\mc I(2r)+c(2r)^{2-n}),\]
which implies that $\mc I(r)+cr^{2-n} \leq 3\mc I(2r)$. Together with $J(s)\leq 2J(r)$, we have
\[\mc I(s)\leq J(s)\leq 2J(r)\leq 6\mc I(2r).\]
\end{proof}

\begin{corollary}\label{cor:small scal for I}
Let $n\geq 3$. Given $C_3>0$ and $\alpha>0$, there exists a constant $\epsilon_0=\epsilon_0(C_3,\alpha,\mc N)>0$ so that if $w\geq 0$ is a $C^2$ function on $\mc A(p;r,\epsilon) (4r\leq \epsilon<\epsilon_0)$ satisfying
\begin{enumerate}
\item $\mc I(s)\leq \alpha\mc I(t)$ for $r\leq s\leq t\leq \epsilon$, where $\mc I(s)=s^{1-n}\int_{\partial\mc B(p;s)} w\, d\mu$;
\item $|\Delta w|\leq C_3w$,
\end{enumerate} 
then $\mc I(s)\leq 6\mc I(2r)$ for $s\in[r,\epsilon]$.
\end{corollary}
\begin{proof}
For $s>t$, set 
\[E(s)=s^{1-n}\big(\int_{\partial\mc B(p;s)}\langle\nu,\nabla w\rangle-\int_{\mc A(p;r,s)}\Delta w\big),\]
where $\nu$ is the unit outward normal vector field of $\partial\mc B(p;s)$. Then by the divergence theorem, there exists a constant $c=c(p)$ such that $E(s)=(cs^{2-n})'$. And a direct computation gives that (cf. \cite{CM02}*{Lemma 2.1})
\begin{equation}\label{eq:derative of Is}
\mc I'(s)=s^{1-n}\int_{\partial\mc B(p;s)}\Big[\langle\nu,\nabla w\rangle+H(x)-\frac{n-1}{s}\Big],
\end{equation}
where $H(x)$ is the mean curvature of $\partial\mc B(p;s)$. Recall that 
\[H(x)=\frac{n-1}{s}+O(s).\]
Hence we have
\begin{align*}
|\mc I'-E|\leq&\ s^{1-n}\int_{\mc B(p;r,s)}|\Delta w|+s^{1-n}\int_{\partial\mc B(p;s)}w(x)\cdot\Big|H(x)-\frac{n-1}{s}\Big|\\
\leq&\ s^{1-n}\int_{r}^s\Big[\int_{\partial \mc B(p;t)}C_2w\,d\mu\Big]\,dt+ C_3s\mc I(s)\\
\leq&\ (C_2\alpha+C_3)s\mc I(s)
\end{align*}
Then by Lemma \ref{lem:ode}, there exists $\epsilon_0=\epsilon_0(n,C_2\alpha+C_3)$ so that if $\epsilon<\epsilon_0$, then $\mc I(s)\leq 6\mc I(2r)$ for $s\in[r,\epsilon]$. Note that $C_3$ depends only on $\mc N$. Then the desired result follows.
\end{proof}

\subsection{Minimal graph on annuli in high dimensions}
Recall that $(M^{n+1},\partial M,g)$ is a closed Riemannian manifold with $3\leq (n+1)\leq 7$ and $\mc N$ is an embedded minimal hypersurface. 

In this subsection, we give a proof of Lemma \ref{lem:maximum for annulus}. The main steps are to show that the difference of minimal graph functions on annulus are the `almost harmonic function' in Corollary \ref{cor:small scal for I}.

We now prove a one-sided Harnack inequality for minimal graph functions on annuli.

\begin{proof}[Proof of Lemma \ref{lem:maximum for annulus}]
By Lemma \ref{lem:minimal graph inequality}, we can take $R_0$ small enough so that 
\begin{align*}\label{eq:laplace for minimal funciton}
|\Delta_\mc Nw|\leq &\ C|\nabla^2w|(|\nabla w|^2+|\nabla v|^2)+C|\nabla^2w|(|\nabla v||\nabla^2v|+|\nabla w|)|v|+\\
&+C(1+|\nabla^2w|+|\nabla^2v|)w+C|\nabla w|\cdot|v|+C|\nabla w|\cdot|\nabla^2v|,
\end{align*}
where the constant $C$ depending only on $M,\mc N$ and $K$. Using (\ref{eq:inequality for vt}) and (\ref{eq:upper bound inequality for vt}), we can rewrite it as
\begin{equation}
|\Delta_\mc N w|\leq C|\nabla^2w|(|\nabla w|^2+|v|^2)+C(1+|\nabla^2w|)w+C|\nabla w|\cdot |v|.
\end{equation}
Recall the gradient and the second order estimates in Lemma \ref{lem:der estimate}, then we have for $x\in\mc B(p;\frac{3}{2}r,R)$,
\[|\nabla w|\leq C(M,\mc N,K,C_1)|w|/|x|, \ \ |\nabla^2w|\leq C(M,\mc N,K,C_1)|w|/|x|^2.\]
Together with $|v|\leq K|x|$ and $|w|\leq C_1|x|^2$, we have
\[|\Delta w|\leq C_4|w|,\]
for $x\in \mc A(p;\frac{3}{2}r,R)$.
\begin{claim}\label{claim:decrease or increase}
There exists a constant $\alpha=\alpha(C_4)$ so that \underline{either} $\mc I(R)\leq \alpha\mc I(2r)$; \underline{or} $\mc I(s)\leq \alpha\mc I(t)$ for all $2r\leq s\leq t\leq R$, where $\mc I(s)=s^{1-n}\int_{\partial\mc B(p;s)}w.$
\end{claim}
\begin{proof}[Proof of Claim \ref{claim:decrease or increase}]
By Lemma \ref{lem:gradient estimates}, there exists a constant $C$ so that 
\[\max_{\partial \mc B(p;2r)}w\leq C\min_{\partial\mc B(p;2r)}w,\ \ \max_{\partial \mc B(p;R)}w\leq C\min_{\partial\mc B(p;R)}w.\]
By Lemma \ref{lem:mp from pde}, we also have
\[\min_{\partial \mc A(p;2r,R)}w\leq C\min_{\mc A(p;2r,R)}w,\ \  \max_{ \mc A(p;2r,R)}w\leq C\max_{\partial \mc A(p;2r,R)}w.\]
From this, we have 
\begin{align*}
\mc I(s)\leq&\ \max_{\mc A(2r,R)}w\leq C(\max_{\partial\mc B(p;2r)}w+\max_{\partial \mc B(p;t)}w)\\
\leq&\ C^2(\mc I(2r)+\mc I(t)).
\end{align*}
Now if $\mc I(R)\geq C^3\mc I(2r)$, then 
\[\min_{\partial\mc B(p;R)}w\geq \mc I(R)/C\geq C^2\mc I(2r)> \min_{\partial\mc B(p;2r)}w,\]
which implies that
\[\min_{\partial \mc B(p;2r)}w=\min_{\partial A(p;2r,R)}w\leq C\min_{\mc A(p;2r,R)}w.\]
Therefore,
\begin{align*}
\mc I(s)&\leq C^2(\mc I(2r)+\mc I(t))\leq C^2(C\min_{\partial\mc B(p;2r)}w+\mc I(t))\\
&\leq C^2(C^2\min_{\mc A(p;2r,R)}w+\mc I(t)) \leq C^4(\mc I(t)+\mc I(t))\leq 2C^4I(t).
\end{align*}
Then the desired result follows by setting $\alpha=2C^4$.
\end{proof}

Let $\alpha$ be the constant in Claim \ref{claim:decrease or increase}. Then either $\mc I(R)\leq \alpha\mc I(3r)$, or $\mc I(s)\leq \alpha \mc I(t)$ for all $4r\leq s\leq t\leq R$. In the latter case, applying Corollary \ref{cor:harnack for equivalent radius}, we have $\mc I(R)\leq 6\mc I(3r)$. Hence we always have $\mc I(R)\leq \alpha \mc I(3r)$.

Now applying Corollary \ref{cor:harnack for equivalent radius}, 
\begin{gather*}
\max_{\partial \mc B(p;R)}w\leq C\min_{\partial\mc B(p;R)}w\leq C\mc I(R),\ \ \mc I(3r)\leq \max_{\partial\mc B(p;3r)}w\leq C\min_{\partial\mc B(p;2r)}w.
\end{gather*}
Then the desired inequality follows.

\end{proof}

\section{On the three dimensional case}\label{sec:three dim}

The main purpose of this section is to develop a one-sided Harnack inequality for minimal graph functions over a minimal surface (Theorem \ref{thm:harnack for minimal graphs on minimal surfaces}).

In this section, let $(M^3,g)$ be a three dimensional compact Riemannian manifold and $\mc N$ be an embedded minimal surface in $M$. For $q\in\mc N$ and $s>0$, denote by $\mc B(p;s)$ the geodesic ball in $\mc N$ such that $\dist_{\mc N}(p,\partial \mc N)>s$. We also denote $\mc A(p;r,s)=\mc B(p;s)\setminus\mc B(p;r)$.

\subsection{Positive functions satisfying the minimal condition}

We also approach Theorem \ref{thm:harnack for minimal graphs on minimal surfaces} by a Harnack inequality for positive functions satisfying additional assumptions.

Denote by $|x-y|$ the distance between $x,y\in\mc N$ and $d(x, A)$ the distance between $x\in \mc N$ and $A\subset \mc N$.

Let $w$ be a $C^1$ function on $\mc B(p;s,r)$. For $t\in[s,r]$, set
\begin{gather*}
 \mc I(p;t)=t^{-1}\int_{\partial\mc B(p;t)}w\,d\mu,\\
 \tau (p;t)=\int _{\partial\mc B(p;t)}\langle\nabla w,\nu\rangle\,d\mu,
 \end{gather*}
 where $\nu=\nu(p;t)$ is the unit outward normal vector field on $\partial\mc B(p;t)$.

\begin{definition}\label{def:minimal condition}
We say that a positive $C^2$ function defined on subset of0 $\mc A(p;r^2,\epsilon)\setminus\bigcup_{j=1}^I \mc B(q_j;r^2)\subset\mc N$ satisfies {\em minimal condition with constant $C$} if the following holds true:
\begin{enumerate}[(a)]
\item\label{minimal condition:quadratic increase} $w$ is positive and $w(x)\leq C( |x|^2+r^2)$ for all $x\in \Xi$, where $|x|=\dist_\mc N(x,p)$;
\item\label{minimal condition:gradient estimate}
\begin{gather*}
\max_{\Xi}|d(x,Q)|\cdot|\nabla\log w(x)|<C,\ \text{ where } Q=\{q_j\};
\end{gather*}
hence the Harnack inequality in Corollary \ref{cor:harnack for equivalent radius} holds true;
\item\label{minimal condition:laplacian eqaution} for all $x\in\Xi$,
\[|\Delta  w| \leq C w+\frac{C w^3}{d^4(x,Q)}.\]
\end{enumerate}
\end{definition}

In this part, we prove the one-sided Harnack inequality in three dimensional manifolds. Moreover, such a result holds true for more general functions on surfaces.
\begin{theorem}\label{thm:harnack with minimal condition}
Let $\mc N$ be a two-dimensional Riemannian surface and $\mc B(p;\epsilon)$ be a geodesic ball. Let $Q=\{q_j\}_{j=1}^I$ a subset of $\mc B(p;r^2,\epsilon r/2)$. Let $w$ be a $C^2$ function on $\Xi=\mc B(p;r^2,\epsilon^2)\setminus \bigcup_{j=1}^I\mc B(q_j;r^2)$ and $\epsilon>r^{1/4}$. Given $C_0>0$, there exists $\epsilon_0=\epsilon_0 (C_0,\mc N)$ and $C=C(C_0,\mc N)$ so that if $\epsilon<\epsilon_0$ and $w$ satisfies minimal condition with constant $C_0$ (see Definition \ref{def:minimal condition}), then we have 
 \[\max_{\partial\mc B(p;\epsilon^2/2)}w\leq C\min_{\partial\mc B(p;\epsilon r)}w.\]
\end{theorem}

We postpone the proof in \S \ref{subsec:proof of harnack for 3 dim}.

\vspace{1em}

Here we give a proposition which is used in the next subsection.
\begin{proposition}\label{prop:good subset and scale}
Let $Q=\{q_j\}_{j=1}^I\subset\mc B(p;4r^2,\epsilon r)$. Suppose that $r<4^{-16I-16}$ and $\epsilon>4^4r^{1/4}$. Then there exist a constant $\theta\in[1/4^{4I+3},1/16)$ and subset $Q'\subset Q$ so that
\begin{enumerate}[(i)]
\item\label{item:containing condition} for any $q\in Q$, there exists $x\in Q'$ satisfying $\mc B(q;r^2)\subset \mc B(x;\theta|x|/4)$ or $\mc B(q;r^2)\subset\mc B(p;\theta \epsilon^{3/4}r^{5/4}/4)$;
\item $\mc B(x;4\theta|x|)\cap\mc B(y;4\theta|y|)=\emptyset$ for any $x,y\in Q'$ with $x\neq y$;
\item $\mc B(x;4\theta|x|)\cap \mc B(p;4\theta \epsilon^{3/4}r^{5/4})=\emptyset$ for any $x\in Q'$,
\end{enumerate} 
where $|x|=\dist(x,p)$.
\end{proposition}
\begin{proof}
We say that $\wti Q$ and $\alpha>0$ satisfy the {\em containing condition} if for any $q\in Q$, there exists $x\in\wti Q$ satisfying $\mc B(q;r^2)\subset \mc B(x;\alpha|x|/4)$ or $\mc B(q;r^2)\subset \mc B(p;\alpha \epsilon^{3/4}r^{5/4}/4)$.

We first note that for any $\alpha\in [1/4^{4I+3},1/16]$, $Q$ and $\alpha$ satisfy the containing condition.

Now we proceed to the proof of the proposition. In the first step, take $\theta_1=1/4^{4I+3}$ and a subset $Q_1\subset Q$ so that 
\begin{itemize}
\item $Q_1$ and $\theta_1$ satisfy the containing condition;
\item for any two different point $x',x''\in Q_1$, $\mc B(x';\alpha|x'|/4)\setminus \mc B(x'';\alpha|x''|/4)\neq \emptyset$, $\mc B(x';\alpha|x'|/4)\setminus\mc B(p;\alpha \epsilon^{3/4}r^{5/4}/4)\neq \emptyset$.
\end{itemize}

If $\theta_1$ satisfies all we need, then we are done. Otherwise, there exist two different points $x,y\in Q_1$ so that $\mc B(x;4\theta|x|)\cap \mc B(y;4\theta|y|)=\emptyset$ or $\mc B(x;4\theta|x|)\cap \mc B(p;4\theta \epsilon^{3/4}r^{5/4})= \emptyset$. In both cases, we can take $\theta_2=4^4\theta_1$ and $Q_2\subset Q_1$ so that $Q_2$ and $\theta_2$ satisfy the above containing condition. 

Note that now the number of the elements in $Q_2$ is less than that of $Q_1$. Hence we can repeat the argument at most $I$ steps. Suppose that we stop at $Q_k$ and $\theta_k$. Then $\theta=\theta_k$ and $Q'=Q_k$ are the desired constant and subset.
\end{proof}

\subsection{The lower bound}
In this subsection, we prove a lower bound of $\mc I(p;\epsilon r)$ if $w$ satisfying the minimal condition. 
\begin{lemma}\label{lem:mp for all annuli}
Suppose that $\epsilon>r^{1/4}$ Then for $\epsilon$ small enough, there exists a constant $C$ so that for any positive function $u$ on $\mc A(q;\rho,R)\subset \mc B(p;\epsilon r)$ satisfying $|\Delta w|\leq C_0w/(r\epsilon^3)$, we have
\[\min_{\partial \mc A(p;\rho,R)}w\leq C\min_{\mc A(p;\rho,R)}w,\ \  \max_{ \mc A(p;\rho,R)}w\leq C\max_{\partial \mc A(p;\rho,R)}w.\]
\end{lemma}
\begin{proof}
Note that 
\begin{align*}
R^2|\Delta w|\leq 4r^2\epsilon^2|\Delta w|\leq C_0wr/\epsilon<w/4.
\end{align*}
Then our desired result follows from Lemma \ref{lem:mp from pde}.
\end{proof}

In the following of this subsection, we always assume that 
\begin{gather*}
\Xi :=\mc A(p;r^2,\epsilon r)\setminus\bigcup_{j=1}^I \mc B(q_j;r^2) \subset \mc N.
\end{gather*}

\begin{proposition}\label{prop:at least a large residue}
Let $w\in C^2(\Xi)$ be a positive function. Suppose that $w$ satisfies (\ref{minimal condition:quadratic increase})(\ref{minimal condition:gradient estimate}) in Definition \ref{def:minimal condition} with constant $C_0$ and
\begin{itemize}
\item $|\Delta w|\leq C_0 w/(r\epsilon^3)$;
\item $\mc I (p;\epsilon r)\leq \tau_0 |\log(\epsilon/r)|$, where $\tau_0=\tau(p;\epsilon r)$.
\end{itemize}
Let $Q=\{q_j\}$ and take $Q'$ and $\theta$ by Proposition \ref{prop:good subset and scale}. 

Then there exists $x_j\in Q'$ such that $\tau(x_j;\theta|x_j|)\geq \tau_0/4^{I+1}$ or $\tau(p;\theta \epsilon^{3/4}r^{5/4})\geq \tau_0/4^{I+1}$.
\end{proposition}
\begin{proof}
We assume that for $y\in Q'$, $\tau(y;\theta|y|)\leq \tau/4^I$. To prove the proposition, it suffices to show that $\tau(p;\theta \epsilon^{3/4}r^{5/4})\geq \tau_0/4^{I+1}$.

By Lemma \ref{lem:number lemma}, we can find $2N+1(\leq 2I+1)$ non-negative intergers $\{k_j\}_{j=0}^{2N}$ such that 
\begin{itemize}
\item $k_{j+1}-k_j\geq 3$;
\item $k_{2j}-k_{2j-1}\leq 10I+1$ for all $j\geq 1$;
\item $Q'\subset \bigcup_{j=1}^{2N}\mc A(p;2^{k_{2j-1}+1}\theta \epsilon^{3/4}r^{5/4},2^{k_{2j}-1}\theta \epsilon^{3/4}r^{5/4})$.
\end{itemize}

We now proceed the desired results by several steps:

{\bf Step 1:} We show that $\tau(p;2^{k_{2N}}\theta \epsilon^{3/4}r^{5/4})\geq \tau_0/2$ and $\mc I(p;2^{k_{2N}}\theta \epsilon^{3/4}r^{5/4})\leq 2\tau_0|\log(\epsilon/r)|$. Indeed, from (\ref{eq:derative of Is}), together with Lemma \ref{lem:mp for all annuli}, we have
\begin{align*}
|\partial_s\mc I(p;s)-\tau_0 s^{-1}|&\leq s^{-1}\int_{\mc A(p;s,\epsilon r)}|\Delta w|+Cs\mc I(p;s)\\
&\leq s^{-1}\int_{\mc A(p;s,\epsilon r)} C_0w/(r\epsilon^3)+Cs\mc I(p;s)\\
&\leq s^{-1}\int_{\mc A(p;s,\epsilon r)}C(\mc I(p;s)+\mc I(p;\epsilon r))/(r\epsilon^3)+Cs\mc I(p;s)\\
&\leq \frac{Cr (\mc I(p;s)+\mc I(p;\epsilon r))}{\epsilon s}.
\end{align*}
Integrating it from $s$ to $\epsilon r$, we get  
\begin{align*}
&|\mc I(p;\epsilon r)-\mc I(p;s)-\tau_0 \log(\epsilon/r)+\tau_0\log s|\\
\leq&\ C_0(\mc I(p;s)+\mc I(p;\epsilon r))(r/\epsilon)|\log(\epsilon/r)|.
\end{align*}
Taking $\epsilon$ small enough, such a inequality implies 
\[\mc I(p;\epsilon r)-\mc I(p;s)-\tau_0\log(\epsilon r/s)\geq -(\mc I(p;s)+\mc I(p;\epsilon r))/10.\]
Hence
\[\mc I(p;s)\leq 2\mc I(p;\epsilon r)\leq 2\tau_0|\log(\epsilon/r)| .\]
By the Harnack inequality from Corollary \ref{cor:harnack for equivalent radius}, we conclude that 
\[w(x)\leq C_1(\mc I(p;\epsilon r)+\mc I(p;2^{k_{2N}}\theta \epsilon^{3/4}r^{5/4}))\leq C_1\tau_0 |\log(\epsilon/r)|,\]
for $x\in \mc A(p;2^{k_{2N}}\theta \epsilon^{3/4}r^{5/4},\epsilon r)$, and $C_1=C_1(\mc N)$. 

Now we show that $\tau(p;s)\geq \tau_0/2$ for $s\in [2^{k_{2N}}\theta \epsilon^{3/4}r^{5/4},\epsilon r]$. Indeed, the divergence theorem gives that
\begin{align*}
\big|\tau(p;s)-\tau_0\big|&\leq\int_{\mc A(p;s,\epsilon r)}|\Delta w|\leq\int_{\mc A(p;s,\epsilon r)}C_0 w/(r\epsilon^3)\\
                         &\leq \int_{\mc A(p;s,\epsilon r)}C_0C_1\tau_0|\log(\epsilon/r)|/(r\epsilon^3)\\
                         &\leq C_0C_1\tau_0(r/\epsilon)|\log(\epsilon/r)|\leq \tau_0/2.
\end{align*}
It follows that $\tau(p;s)\geq \tau_0/2$.

{\bf Step 2:} $\mc I(p;2^{k_{2N-1}}\theta \epsilon^{3/4}r^{5/4})\leq C_1\tau_0|\log(\epsilon/r)|$ and $\tau(p;2^{k_{2N-1}}\theta \epsilon^{3/4}r^{5/4})\geq \tau_0/4 $.

Denote by $\Omega_{j}$ the domain $\mc A(p;2^{k_{j-1}}\theta \epsilon^{3/4}r^{5/4},2^{k_{j}}\theta \epsilon^{3/4}r^{5/4})\setminus \bigcup_{y\in Q'}\mc B(y;\theta|y|)$.

Since $k_{2N}-k_{2N-1}\leq 10I+1$, then by the Harnack inequality from Corollary \ref{cor:harnack for equivalent radius}, there exists $C_1'=C_1'(\mc N,I)$ satisfying
\[w(x)\leq C'_1\tau_0|\log(\epsilon/r)|, \text{ for all } x\in\Omega_{2N}. \]

Now we use this upper bound of $u$ to estimate $\tau(p,2^{k_{2N-1}}\theta \epsilon^{3/4}r^{5/4})$. By divergence theorem,
\begin{align*}
&\Big|\tau(p;2^{k_{2N}}\theta \epsilon^{3/4}r^{5/4})-\tau(p;2^{k_{2N-1}}\theta \epsilon^{3/4}r^{5/4})-\sum_{y\in Q'\cap\mc A(p;2^{k_{2N-1}}\theta \epsilon^{3/4}r^{5/4},2^{k_{2N}}\theta \epsilon^{3/4}r^{5/4})}\tau(y;\theta |y|)\Big|\\
\leq &\int_{\Omega_{2N}}|\Delta w|\leq\ \int_{\Omega_{2N}}C_0C_1'\cdot \tau_0|\log(\epsilon/r)|/(r\epsilon^3)\leq\tau_0/4. 
\end{align*}

Note that by assumptions, $\tau(y;\theta|y|)\leq \tau_0/4^I$ for any $y\in Q'$. Thus we conclude that $\tau(p;2^{k_{2N-1}}\theta \epsilon^{3/4}r^{5/4})\geq \tau_0/4 $.
Thus we finish the proof of Step 2.

\vspace{1em}
Running the argument in Step 1 again, we can prove that $\tau(p;2^{k_{2N-2}}\theta \epsilon^{3/4}r^{5/4})\geq \tau_0/8$ and $\mc I(p;2^{k_{2N-2}}\theta \epsilon^{3/4}r^{5/4})\leq C_1\tau_0 |\log (\epsilon/r)|$. Then repeating the argument in Step 2, we obtain $\tau(p;2^{k_{2N-3}}\theta \epsilon^{3/4}r^{5/4})\geq \tau_0/16$ and $\mc I(p;2^{k_{2N-3}}\theta \epsilon^{3/4}r^{5/4})\leq C_1\tau_0 |\log(\epsilon/r)|$.

By induction, we conclude that $\tau(p;\theta \epsilon^{3/4}r^{5/4})\geq \tau_0/4^I$. Thus we complete the proof of Proposition \ref{prop:at least a large residue}.
\end{proof}

\begin{proposition}\label{prop:good sequence of Q from R to R3/4}
Let $w\in C^2(\Xi)$. Suppose that $w$ satisfies minimal condition with constant $C_0$ and $\mc I (p;\epsilon r)\leq \tau_0|\log(\epsilon/r)|$, where $\tau_0=\tau(p;\epsilon r)$. Take $Q'\subset Q$ and $\theta$ by Proposition \ref{prop:good subset and scale}. Then either 
\begin{equation*}
\tau(y,\theta|y|)\leq \tau_0/4^{I+1}, \forall y\in Q',
\end{equation*}
or there exist $c_0=c_0(C_0,\mc N,I)$, $c_1=c_1(C_0,\mc N,I)$, a sequence $y_1,y_2,...,y_k\in Q$ and $\theta_1,\theta_2,...,\theta_{k+1}\in (1/4^{4I+3},1/16)$ such that 
\begin{enumerate}[(i)]
\item\label{good sequence:suitable y_1} $\mc I(p;\epsilon r)-c_0\mc I(y_1;\theta_1|y_1|)\geq c_0\tau_0 (\log \epsilon r-\log(\theta_1|y_1|)-c_1)$;
\item\label{good sequence:controlled one by one} for $j\leq k-1$, $\mc I(y_j;\theta_j|y_j-y_{j-1}|)-c_0\mc I(y_{j+1};\theta_{j+1}|y_{j+1}-y_j|)\geq c_0\tau_0 \big(\log(\theta_j|y_j-y_{j-1}|)-\log(\theta_{j+1}|y_{j+1}-y_j|)-c_1\big)$, where $y_0=p$;
\item\label{good sequence:continue to R3/4} 
\begin{align*}
&\mc I(y_k;\theta_k|y_k-y_{k-1}|)-c_0\mc I(y_k;\theta_{k+1}\epsilon^{3/4}r^{5/4})\\
\geq& c_0\tau_0 \big(\log(\theta_k|y_k-y_{k-1}|)-\log(\theta_{k+1}\epsilon^{3/4}r^{5/4})-c_1\big).
\end{align*}
\end{enumerate}
\end{proposition}
\begin{proof}
Without loss of generality, we assume that there exist $y'\in Q'$ so that 
\begin{equation}\label{eq:all small condition}
\tau(y';\theta|y'|)\geq \tau_0/4^{I+1}.
\end{equation}
Set $\theta_1=\theta$ and take $y_1\in Q'$ so that $\tau(y_1;\theta_1|y_1|)\geq \tau_0/4^{I+1}$ and 
\[\tau(y;\theta_1 |y|)< \tau_0/4^{I+1},\ \forall y\in Q' \text{ with } |y|>|y_1|.\]
We remark that such $y_1$ is well-defined by (\ref{eq:all small condition}).

\vspace{0.7em}
{\bf Step A:} We first show that there exists $c_0=c_0(C_0,\mc N,I)$ and $c_1=c_1(C_0,\mc N,I)$ so that (\ref{good sequence:suitable y_1}) is satisfied.
 
By Lemma \ref{lem:number lemma}, we can find $2N+1(\leq 2I+1)$ non-negative intergers $\{k_j\}_{j=0}^{2N}$ such that 
\begin{itemize}
\item $k_{j+1}-k_j\geq 3$;
\item $k_{2j}-k_{2j-1}\leq 10I+1$ for all $j\geq 1$;
\item $Q'\subset \bigcup_{j=1}^{2N} \mc A(p;2^{k_{2j-1}+1}\theta_1 \epsilon^{3/4}r^{5/4},2^{k_{2j}-1}\theta_1 \epsilon^{3/4}r^{5/4})$.
\end{itemize} 
Hence there exists $j'\leq N$ such that $y_1\in \mc A(p;2^{k_{2j'-1}+1}\theta_1 \epsilon^{3/4}r^{5/4},2^{k_{2j'}-1}\theta_1 \epsilon^{3/4}r^{5/4})$. By the choice of $y_1$, 
\[\tau(y;\theta_1 |y|)\leq \tau_0/4^{I+1},\ \forall y\in Q'\cap \mc A(p;2^{k_{2j'}}\theta_1 \epsilon^{3/4}r^{5/4},\epsilon r).\]

Then by the induction in Step 1 and 2 in Proposition \ref{prop:at least a large residue}, we conclude that 
\begin{gather}
w(x)\leq C_1'\tau_0 |\log(\epsilon/r)|, \ \ \text{ for all } x\in \mc A(p;2^{k_{2j'-1}}\theta_1 \epsilon^{3/4}r^{5/4},\epsilon  r)\setminus\bigcup_{y\in Q'}\mc B(y,\theta_1|y|),\\
\label{eq:large tau from 2j' to 2N}\tau(p;2^{k_j}\theta_1 \epsilon^{3/4}r^{5/4})\geq \tau_0/4^{I+1}, \ \ \text{ for all } 2j'\leq j\leq 2N,
\end{gather}
where $C_1'=C_1'(\mc N,I)$.

Using the estimates of $\tau$ for $2j'\leq j\leq 2N$, (\ref{eq:derative of Is}) together with divergence theorem implies that for $s\in [2^{k_{2j}}\theta_1\epsilon^{3/4}r^{5/4},2^{k_{2j+1}}\theta_1\epsilon^{3/4}r^{5/4}]$,
\begin{align*}
|\partial_s\mc I(p;s)-\tau(p;2^{k_{2j+1}}\theta_1\epsilon^{3/4}r^{5/4}) s^{-1}|&\leq\ s^{-1}\int_{\mc A(p;s,2^{k_{2j+1}}\theta_1\epsilon^{3/4}r^{5/4})}|\Delta w|+Cs\mc I(s)\\
&\leq s^{-1}\int_{\mc A(p;s,2^{k_{2j+1}}\theta_1\epsilon^{3/4}r^{5/4})} C_0 w^3/d^4(x,Q)+Cs\mc I(s).
\end{align*}
Note that in this case, $d(x,Q)\geq \epsilon^{3/4}r^{5/4}$. Then the inequality becomes
\begin{align*}
|\partial_s\mc I(p;s)-\tau(p;2^{k_{2j+1}}\theta_1\epsilon^{3/4}r^{5/4}) s^{-1}|&\leq s^{-1}\int_{\mc A(p;s,2^{k_{2j+1}}\theta_1\epsilon^{3/4}r^{5/4})} C_0 w/(r\epsilon^3)+Cs\mc I(s)\\
&\leq s^{-1}C_0(r/\epsilon)(\mc I(p;s)+\mc I(p;2^{k_{2j+1}}\theta_1\epsilon^{3/4}r^{5/4}))+Cs\mc I(s)\\
&\leq C_0(\mc I(p;s)+\mc I(p;2^{k_{2j+1}}\theta_1\epsilon^{3/4}r^{5/4}))r/(\epsilon s).
\end{align*}
Integrating it from $s$ to $2^{k_{2j+1}}\theta_1\epsilon^{3/4}r^{5/4}$, we get  
\begin{align*}
&|\mc I(p;2^{k_{2j+1}}\theta_1\epsilon^{3/4}r^{5/4})-\mc I(p;s)-\tau(p;2^{k_{2j+1}}\theta_1\epsilon^{3/4}r^{5/4}) \log (2^{k_{2j+1}}\theta_1\epsilon^{3/4}r^{5/4}/s)|\\
\leq&\ C_0(\mc I(p;s)+\mc I(p;2^{k_{2j+1}}\theta_1\epsilon^{3/4}r^{5/4}))(r/\epsilon)|\log(\epsilon/r)|\\
\leq& (\mc I(p;s)+\mc I(p;2^{k_{2j+1}}\theta_1\epsilon^{3/4}r^{5/4}))/10.
\end{align*}
Together with (\ref{eq:large tau from 2j' to 2N}), this implies that for $j'\leq j\leq N-1$,
\[\mc I(p;2^{k_{2j+1}}\theta_1\epsilon^{3/4}r^{5/4})-\frac{1}{2}\mc I(p;2^{k_{2j}}\theta_1\epsilon^{3/4}r^{5/4})\geq\frac{\tau_0}{4^{I+2}}(\log (2^{k_{2j+1}}\theta_1\epsilon^{3/4}r^{5/4})-\log (2^{k_{2j}}\theta_1\epsilon^{3/4}r^{5/4})) .\]
A similar argument for $[2^{k_{2N}}\theta_1\epsilon^{3/4}r^{5/4},r]$ gives that 
\begin{equation}\label{eq:estimate from R to outside circle}
\mc I(p;\epsilon r)-\frac{1}{2}\mc I(p;2^{k_{2N}}\theta_1\epsilon^{3/4}r^{5/4})\geq\frac{\tau_0}{4^{I+2}}(\log(\epsilon r)-\log (2^{k_{2N}}\theta_1\epsilon^{3/4}r^{5/4})).
\end{equation}
Recall the Harnack inequality from Corollary \ref{cor:harnack for equivalent radius} gives that 
\begin{gather*}
\mc I(p;2^{k_{2j}}\theta_1 \epsilon^{3/4}r^{5/4})\geq \gamma\mc I(p;2^{k_{2j-1}}\theta_1 \epsilon^{3/4}r^{5/4}),
\end{gather*}
for $1\leq j\leq N$ and $\gamma=\gamma(\mc N,I)<1/2$. Particularly,
\[\mc I(p;2^{k_{2j'}}\theta \epsilon^{3/4}r^{5/4})\geq \gamma\mc I(y_1;\theta_1|y_1|).\]
Then putting them together with suitable coefficients (see Appendix \ref{sec:appendix:suitable coeff} for details), we have
\begin{equation}\label{eq:result by adding with suitable coeff}
\mc I(p;\epsilon r)-c_0\mc I(y_1;\theta_1|y_1|)\geq c_0\tau_0 (\log\epsilon  r-\log(\theta_1|y_1|)-c_1)
\end{equation}
by setting $c_0=\gamma^{2N}/r^{I+2}$ and $c_1=\gamma^{-2N}I(10I+1)\log 2$.

Thus we complete Step A.

\vspace{0.7em}

{\bf Step B:} We now repeat the process in Step A to construct $\{y_j\}$ and $\{\theta_j\}$.

Suppose that we have defined $y_j$ and $\theta_j$, then by Proposition \ref{prop:good subset and scale}, there exist a constant $\theta_{j+1}\in (1/4^{I+3},1/16)$ and a subset $Q_{j+1}\subset Q\cap\mc B(y_j;\theta_j|y_j-y_{j-1}|)$ such that 
\begin{itemize}
\item for any $y\in Q\cap\mc B(y_j;\theta_j|y_j-y_{j-1}|)$, there exists $y'\in Q_{j+1}$ satisfying $\mc B(y;r^2)\subset \mc B(y';\theta_{j+1}|y'-y_j|/4)$ or $\mc B(y;r^2)\subset \mc B(y_j;\theta_{j+1} \epsilon^{3/4}r^{5/4}/4)$;
\item $\mc B(x';4\theta_{j+1}|x'-y_j|)\cap B(x'';4\theta_{j+1}|x''-y_j|)=\emptyset$ for any $x',x''\in Q_{j+1}$ with $x'\neq x''$;
\item $\mc B(x';4\theta_{j+1}|x'-y_j|)\cap\mc B(y_j;4\theta_{j+1} \epsilon^{3/4}r^{5/4})=\emptyset$ for any $x'\in Q_{j+1}$.
\end{itemize} 
Note that for $x\in \mc A(y_j;\theta_{j+1} \epsilon^{3/4}r^{5/4},\theta_j|y_j|)\setminus \bigcup_{y\in Q_{j+1}}\mc B(y;\theta_{j+1}|y-y_j|)$, we always have some constant $C_I$ depending only on $I$ satisfying
\[d(x, Q)\geq \epsilon^{3/4}r^{5/4}/C_I,\]
which implies that
\[|\Delta w|\leq \frac{C_0C_Iw}{r\epsilon^3}.\]
\begin{claim}\label{claim:prepare to construct yj1}
Either $\tau(y;\theta_{j+1}|y-y_j|)\leq \tau(y_j;\theta_j|y_j-y_{j-1}|)/4^{I+1}$ for all $y\in Q_{j+1}$ ($y_0=p$), or there exists $x_{j+1}\in Q_{j+1}$ so that 
\begin{equation}\label{eq:exist of good xj1}
 \tau(x_{j+1};\theta_{j+1}|x_{j+1}-y_j|)\geq \tau(y_j;\theta_j|y_j-y_{j-1}|)/4^{I+1}.
 \end{equation}
\end{claim}
\begin{proof}[Proof of Claim \ref{claim:prepare to construct yj1} ]
The proof here is similar to that of Proposition \ref{prop:at least a large residue} with minor modification. Indeed, it follows by replacing $\mc A(p;r^2,\epsilon r)$ with $\mc A(y_j;r^2,\theta_{j}|y_j-y_{j-1}|)$.
\end{proof}

Whenever $Q_{j+1}\neq \emptyset$ and there exists $x_{j+1}$ so that (\ref{eq:exist of good xj1}) is satisfied, then we can define $y_{j+1}$ to be the element in $Q_{j+1}$ so that $\tau(y_{j+1};\theta_{j+1}|y_{j+1}-y_j|)\geq \tau(y_{j};\theta_{j}|y_{j}-y_{j-1}|)/4^{I+1}$ and 
\[\tau(y;\theta_{j+1}|y-y_j|)< \tau(y_{j};\theta_{j}|y_{j}-y_{j-1}|)/4^{I+1}, \forall y\in Q_{j+1} \text{ with } |y-y_j|>|y_{j+1}-y_j|.\]

Since $y_j\in\mc B(y_j;\theta_j|y_j-y_{j-1}|)\setminus\mc B(y_{j+1};\theta_{j+1}|y_{j+1}-y_j|)$, then such a sequence of $y_j$ is finite, that is, there exists $0<k<I+1$ so that $Q_{k+1}=\emptyset$ or $\tau(y;\theta_{k+1}|y-y_k|)\leq \tau(y_k;\theta_j|y_k-y_{k-1}|)/4^{I+1}$ for all $y\in Q_{k+1}$ ($y_0:=p$).

Using the argument in Step A, by replacing $\mc A(p;r^2,\epsilon r)$ with $\mc A(y_j;r^2,\theta_{j}|y_j-y_{j-1}|)$, we can see that (\ref{good sequence:controlled one by one}) is satisfied. We leave the details to readers.

\vspace{0.7em}
{\bf Step C:} We now prove that such a sequence $\{y_j\}$ satisfies (\ref{good sequence:continue to R3/4}).

By the construction in Step B, we know that either $Q_{k+1}=\emptyset$ or $\tau(y;\theta_{k+1}|y-y_k|)\leq \tau(y_k;\theta_k|y_k-y_{k-1}|)/4^{I+1}$ for all $y\in Q_{k+1}$ ($y_0:=p$).
\begin{claim}\label{claim:step C:someone large}
In both case, $\tau(y_k;\theta_{k+1}\epsilon^{3/4}r^{5/4})\geq \tau(y_k;\theta_{k}|y_k-y_{k-1}|)/4^{I+2}$.
\end{claim}
\begin{proof}[Proof of Claim \ref{claim:step C:someone large}]
Note that for $x\in \mc B(y_k;\theta_{k+1}\epsilon^{3/4} r^{5/4},\theta_k|y_k-y_{k-1}|)\setminus \bigcup_{y\in Q_{k+1}}\mc B(y;\theta_{k+1}|y-y_k|)$, we always have
\[d(x, Q)\geq \epsilon^{3/4}r^{5/4}/C_I,\]
for $C_I$ depending only on $I$. This implies that
\[|\Delta w|\leq \frac{CC_Iw}{r\epsilon^3}.\]
Then the desired results follows from Proposition \ref{prop:at least a large residue} by replacing $\mc A(p;r^2,\epsilon r)$ with $\mc A(y_k;r^2,\theta_{k}|y_k-y_{k-1}|)$
\end{proof}
Then the following is similar to that of Step A. For completeness of the proof, we give more details here.

By Lemma \ref{lem:number lemma}, we can find $2N\leq 2I+1)$ non-negative integers $\{l_j\}_{j=0}^{2N}$ ($l_0=0$) such that 
\begin{itemize}
\item $l_{j+1}-l_j\geq 3$;
\item $l_{2j}-l_{2j-1}\leq 10I+1$ for all $j\geq 1$;
\item $Q_{k+1}\subset \bigcup_{j=1}^{2N} \mc A(p;2^{l_{2j-1}+1}\theta_1 \epsilon^{3/4}r^{5/4},2^{l_{2j}-1}\theta_1 \epsilon^{3/4}r^{5/4})$.
\end{itemize}
Then the argument in Step A implies that for $j\geq0$,
\begin{align*}
&\ \mc I(y_k;2^{l_{2j+1}}\theta_{k+1}\epsilon^{3/4}r^{5/4})-\frac{1}{2}\mc I(y_k;2^{l_{2j}}\theta_{k+1}\epsilon^{3/4}r^{5/4})\\
\geq&\ \frac{\tau(y_k;\theta_k|y_k-y_{k-1}|)}{4^{I+2}}\big[\log (2^{l_{2j+1}}\theta_{k+1}\epsilon^{3/4}r^{5/4})-\log (2^{l_{2j}}\theta_{k+1}\epsilon^{3/4}r^{5/4})\big],
\end{align*}
and
\begin{align*}
&\ \mc I(y_k;\theta_k|y_k-y_{k-1}|)-\frac{1}{2}\mc I(y_k;2^{l_{2N}}\theta_{k+1}\epsilon^{3/4}r^{5/4})\\
\geq&\ \frac{\tau(y_k;\theta_k|y_k-y_{k-1}|)}{4^{I+2}}\big(\log (\theta_k|y_k-y_{k-1}|)-\log(2^{l_{2N}}\theta_{k+1}\epsilon^{3/4}r^{5/4})\big).
\end{align*}
Recall the Harnack inequality from Corollary \ref{cor:harnack for equivalent radius} gives that 
\begin{gather*}
\mc I(y_k;2^{l_{2j}}\theta_{k+1} \epsilon^{3/4}r^{5/4})\geq \gamma\mc I(y_k;2^{l_{2j-1}}\theta_{k+1} \epsilon^{3/4}r^{5/4}),
\end{gather*}
for $1\leq j\leq N$ and $\gamma=\gamma(\mc N,I)<1/2$. Then putting them together with suitable coefficients (also, see Appendix \ref{sec:appendix:suitable coeff} for a similar process), we have
\begin{align*}
&\mc I(y_k;\theta_k|y_k-y_{k-1}|)-c_0\mc I(y_k;\theta_{k+1}\epsilon^{3/4}r^{5/4})\\
\geq& c_0\tau_0 \big[\log(\theta_k|y_k-y_{k-1}|)-\log(\theta_{k+1}\epsilon^{3/4}r^{5/4})-c_1\big].
\end{align*}
by setting $c_0=\gamma^{2N}/r^{I+2}$ and $c_1=\gamma^{-2N}I(10I+1)\log 2$.

This is the desired inequality.
\end{proof}

\begin{lemma}\label{lem:estimate of taulogR}
Let $w$ be the function in Proposition \ref{prop:good sequence of Q from R to R3/4}. Then there exists a constant $C_5=C_5(C_1,\mc N,I)$ so that
\[\tau_0\log(\epsilon/r)\leq C_5\mc I(p;\epsilon r).\]
\end{lemma}
\begin{proof}
We first assume that 
\begin{equation*}
\tau(y,\theta|y|)\leq \tau_0/4^{I+1}, 
\end{equation*}
for all $y\in Q'$. Then the argument in Step A and Step C in Proposition \ref{prop:good sequence of Q from R to R3/4} gives that 
\begin{align*}
\mc I(p;\epsilon r)-c_0\mc I(p;\theta \epsilon^{3/4}r^{5/4})\geq&\ c_0\tau_0\big(\log \epsilon r-\log (\theta \epsilon^{3/4}r^{5/4})-c_1\big)\\
\geq&\ c_0\tau_0(\frac{1}{4}\log(\epsilon/r)-c_1).
\end{align*}
It follows that $\tau_0\log(\epsilon/r)\leq C_5\mc I(p;\epsilon r)$ for some constant $C_5$ depending only on $c_1,c_0$. Note that $c_0$ and $c_1$ depend only on $C_1,\mc N,I$. This is exactly the desired result.

Now we consider that there is $y\in Q'$ satisfying
\[ \tau(y,\theta|y|)> \tau_0/4^{I+1}.\]
Then by Proposition \ref{prop:good sequence of Q from R to R3/4}, there exist $c_0$ and $c_1$ depending only on the constant $C_0$ (in Proposition \ref{prop:good sequence of Q from R to R3/4}) and $I$, a sequence $y_1,y_2,...,y_k\in Q$ and $\theta_1,\theta_2,...,\theta_{k+1}\in (1/4^{4I+3},1/16)$ such that 
\begin{enumerate}[(i)]
\item $\mc I(p;\epsilon r)-c_0\mc I(y_1;\theta_1|y_1|)\geq c_0\tau_0 (\log\epsilon r-\log(\theta_1|y_1|)-c_1)$;
\item for $j\leq k-1$, $\mc I(y_j;\theta_j|y_j-y_{j-1}|)-c_0\mc I(y_{j+1};\theta_{j+1}|y_{j+1}-y_j|)\geq c_0\tau_0 \big(\log(\theta_j|y_j-y_{j-1}|)-\log(\theta_{j+1}|y_{j+1}-y_j|)-c_1\big)$, where $y_0=p$;
\item
\begin{align*}
&\mc I(y_k;\theta_k|y_k-y_{k-1}|)-c_0\mc I(y_k;\theta_{k+1}\epsilon^{3/4}r^{5/4})\\
\geq& c_0\tau_0 \big[\log(\theta_k|y_k-y_{k-1}|)-\log(\theta_{k+1}\epsilon^{3/4}r^{5/4})-c_1\big].
\end{align*}
\end{enumerate}
By adding them together with suitable coefficients like this:
\begin{align*}
&\mc I(p;\epsilon r)-c_0^{k+1}\mc I(y_k;\theta_{k+1}\epsilon^{3/4}r^{5/4})\\
=&\ \mc I(p;\epsilon r)-c_0\mc I(y_1;\theta_1|y_1|)+c_0^k\Big[\mc I(y_k;\theta_k|y_k-y_{k-1}|)-c_0\mc I(y_k;\theta_{k+1}\epsilon^{3/4}r^{5/4})\Big]+\\
&+\sum_{j=1}^{k-1}c_0^j\Big[\mc I(y_j;\theta_j|y_j-y_{j-1}|)-c_0\mc I(y_{j+1};\theta_{j+1}|y_{j+1}-y_j|)\Big]+\\
\geq&\ c_0\tau_0 (\log\epsilon r-\log(\theta_1|y_1|)-c_1)+c_0^{k+1}\tau_0 \big(\log(\theta_k|y_k-y_{k-1}|)-\log(\theta_{k+1}\epsilon^{3/4}r^{5/4})-c_1\big)\\
&+\sum_{j=1}^{k-1}c_0^{j+1}\tau_0 \big(\log(\theta_j|y_j-y_{j-1}|)-\log(\theta_{j+1}|y_{j+1}-y_j|)-c_1\big)\\
\geq&\ c_0^{I+1}\tau_0\big(\log\epsilon r-\log(\theta_{k+1}\epsilon^{3/4}r^{5/4})-c_1(I+1)\big)\\
\geq&\ c_0^{I+1}\tau_0\big(\frac{1}{4}\log(\epsilon/r)-\log\theta_{k+1}-c_1(I+1)\big).
\end{align*}
Then the desired result follows from the argument in the first case.
\end{proof}

\subsection{One-sided Harnack inequality for minimal graph functions on minimal surfaces}\label{subsec:proof of harnack for 3 dim}
In this subsection, we first use the obtained results in previous subsection to prove Theorem \ref{thm:harnack with minimal condition} and then give a proof of Theorem \ref{thm:harnack for minimal graphs on minimal surfaces}. 

\begin{proof}[Proof of Theorem \ref{thm:harnack with minimal condition}]
By the assumptions of $Q$, it follows that for $x\in\mc B(p;\epsilon r,\epsilon^2)$, 
\[\dist(x,Q)\geq |x|/2,\]
which implies that 
\begin{equation}\label{eq:laplace inequality}
|\Delta w|\leq C_0w(1+\frac{r^4}{|x|^4}).
\end{equation}
\begin{claim}\label{claim:mp on outside annulus}
There exists a constant $C_6$ depending only on $\mc N,C_0$ so that 
\[w(x)\leq C_6(\mc I(p;\epsilon^2)+\mc I(p;\epsilon r)),\]
for all $x\in\mc A(p;\epsilon r,\epsilon^2)$.
\end{claim}
\begin{proof}[Proof of Claim \ref{claim:mp on outside annulus}]
Note that for $x\in \mc A(p;2C_0\epsilon r,\epsilon^2)$, 
\[\epsilon^4|\Delta w|\leq C_0w\epsilon^4(1+\frac{1}{C_0^4\epsilon^4})<w/4.\]
By virtue of Lemma \ref{lem:mp from pde}, there exists a constant $C_7$ so that for $x\in\mc A(p;2C_0\epsilon r,\epsilon^2)$,
\begin{equation}\label{eq:mp for large outside annulus}
w(x)\leq C_7\max_{\partial \mc A(p;2C_0\epsilon r,\epsilon^2)}w\leq C_7(\mc I(p;2C_0\epsilon r)+\mc I(p;\epsilon^2)).
\end{equation}
Recall that for $x\in\mc A(p;\epsilon r,2C_0\epsilon r)$, by Corollary \ref{cor:harnack for equivalent radius},
\begin{equation}\label{eq:mp for small outside annulus}
w(x)\leq C'\mc I(p;\epsilon r),
\end{equation}
for some constant $C'=C'(\mc N,C_0)$. Then Claim \ref{claim:mp on outside annulus} follows from (\ref{eq:mp for large outside annulus}) and (\ref{eq:mp for small outside annulus}).
\end{proof}
Then for $s\in(\epsilon r,\epsilon^2)$,
\begin{align*}
|\partial_s\mc I(p;s)-\tau_0 s^{-1}|&\leq s^{-1}\int_{\mc A(p;\epsilon r,s)}|\Delta w|\leq s^{-1}\int_{\mc A(p;\epsilon r,s)}C_0w\cdot (1+\frac{r^4}{|x|^4})\\
                 &\leq s^{-1}\int_{\mc A(p;\epsilon r,s)}C_6(\mc I(p;\epsilon^2)+\mc I(p;\epsilon r))(1+\frac{r^4}{|x|^4})\\
                 &\leq C_6(\mc I(p;\epsilon^2)+\mc I(p;\epsilon r))(s+\frac{r^2}{\epsilon^2s}). 
\end{align*}
Integrating it from $\epsilon r$ to $\epsilon^2$, we get   
\begin{align*}
|\mc I(p;\epsilon^2)-\mc I(p;\epsilon r)-\tau_0 \log(\epsilon/r)|&\leq C_6(\mc I(p;\epsilon r)+\mc I(p;\epsilon^2))(s^2+\frac{r^2}{\epsilon^2}|\log(\epsilon/r)|)\\
&\leq (\mc I(p;\epsilon r)+\mc I(p;\epsilon^2))/10,
\end{align*}
which implies that 
\begin{equation}\label{eq:three dim:from R^2 to R}
\mc I(p;\epsilon^2)\leq 2\mc I(p;\epsilon r)+2\tau_0\log(\epsilon/r).
\end{equation}
Now \underline{if $\tau_0\log(\epsilon/r)\leq \mc I(p;\epsilon r)$}, then we have 
\[\mc I(p;\epsilon^2)\leq 4\mc I(p;\epsilon r).\]
\underline{If $\mc I(p;\epsilon r)\leq\tau_0\log(\epsilon/r)$}, then applying Lemma \ref{lem:estimate of taulogR}, together with (\ref{eq:laplace inequality}), we have
\[\tau_0\log (\epsilon/r)\leq C_5\mc I(p;\epsilon r),\]
and it follows that
\begin{equation}\label{eq:bound R2 by R}
\mc I(p;\epsilon^2)\leq 2C_5\mc I(p;\epsilon r).
\end{equation}

In both cases, (\ref{eq:bound R2 by R}) always holds true. Then the desired inequality follows from Corollary \ref{cor:harnack for equivalent radius}.

\end{proof}

Now we are ready to prove our main result in this section:
\begin{proof}[Proof of Theorem \ref{thm:harnack for minimal graphs on minimal surfaces}]
By Lemma \ref{lem:minimal graph inequality}, we can take $\epsilon$ small enough so that 
\begin{align*}
|\Delta_\mc Nw|\leq &\ C|\nabla^2w|(|\nabla w|^2+|\nabla v|^2)+C|\nabla^2w|(|\nabla v||\nabla^2v|+|\nabla w|)|v|+\\
&+C(1+|\nabla^2w|+|\nabla^2v|)w+C|\nabla w||v|+C|\nabla w|\cdot|\nabla^2v|,
\end{align*}
where $C=C(M,\mc N,K)$. Recall that $(v,u)$ is a strong $(C_1(|x|^2+r^2),K)$-pair. Therefore,
\[|\nabla v(x)|+|\nabla^2 v(x)|\leq K|v(x)|\leq K^2|x|.\]
Then the inequality becomes
\begin{align*}
|\Delta_\mc Nw|\leq &\ C|\nabla^2w|(|\nabla w|^2+|x|^2)+C(1+|\nabla^2w|)w+C|\nabla w||x|,
\end{align*}
By virtue of the gradient and the second order estimates in Lemma \ref{lem:der estimate}, then we have for $x\in \Xi$,
\[|\nabla w|\leq C'|w|/d(x,Q), \ \ |\nabla^2w|\leq C'|w|/d^2(x,Q),\]
where $Q=\bigcup_{j=1}^I\{q_j\}\cup\{p\}$. Taking it back, we obtain
\[|\Delta w|\leq C'|w|+\frac{C'w^3}{d^4(x,Q)}.\]

Therefore, $w$ satisfies all the conditions in Theorem \ref{thm:harnack with minimal condition}, which implies that there exist $C_6=C_6(M,\mc N,K,I,C_0)$ so that
 \[\max_{\partial \mc B(p;\epsilon)}w\leq C_6\min_{\partial \mc B(p;\sqrt\epsilon r)} w,\]
which is exactly the desired inequality.
\end{proof}

\appendix
\section{Minimal graph functions}\label{sec:appendix:minimal graph function}
In this section, we let $\mc N$ be a two-sided, embedded minimal hypersurface possibly with boundary in $(M^{n+1},g)$. Denote by $\n$ the unit normal vector field on $\mc N$. Then there exists a local foliation $\{\mc N_s\}$ around $\mc N$ by the level set of the distance function to $\Sigma$, where
\[\Sigma_s=\{\mathrm{exp}_ps\n:p\in \Sigma\}.\]
Here $\n$ is the unit normal vector of $\Sigma$.

An embedded hypersurface $\Sigma$ is said to be a {\em graph over $\mc N$ with function $u$} if the exponential map $\mathrm{exp}(\cdot,u):\mc N\rightarrow\Sigma$ is a diffeomorphism, where $\mathrm{exp}(p,u)=\mathrm{exp}_p(u\n(p))$.

Let $\nabla^s$ be the connection on $\mc N_s$. We will write $\nabla$ with no ambiguity. Denote by $\pi$ the projection to $\mc N$. Then given a function on $\mc N$, $f$ can also be regarded as a function on $\mc N_s$ by defining 
\[\wti f(\mathrm{exp}_x(s\n)):=f(x), \ \ \forall x\in\mc N.\]
 
Note that $\nabla\wti f|_{\mc N_s}$ is the extension of $\nabla f$ by parallel moving. 

Let $d$ be the oriented distance function to $\mc N$. Then $\nabla d$ is the unit normal vector field on $\mc N_s$, which is an extension of $\n$. For $p\in\Sigma\cap \mc N_s$, let $\{e_i\}$ be an orthonormal base of $T_p\mc N_s$. Then $\{e_i+\langle \nabla u,e_i\rangle\nabla d\}$ is a base of $T_p\Sigma$. It follows that $\n_\Sigma=(\nabla d-\nabla u)/\sqrt{1+|\nabla u|^2}$ is the unit normal vector field of $\Sigma$. Naturally, such $\n_\Sigma$ can be extend to $\wti \n_\Sigma$ in a neighborhood of $\mc N$ by parallel moving.

Now let $X$ be a vector field around $\mc N$ so that $\nabla_{\nabla d}X=0$. Recall that $\pi$ is the projection to $\mc N$. Then we have the following:
\begin{lemma}\label{lem:diff of div}
There exist $\delta=\delta(M,\mc N)$ and $C=C(M,\mc N)$ so that for $|h|,|s|<\delta$, 
\[|\mathrm{div}_{\mc N_h}X(\pi^{-1}(x)\cap \mc N_h)-\mathrm{div}_{\mc N_s}X(\pi^{-1}(x)\cap \mc N_s)|\leq C((|X|+|\nabla X|)|_\mc N+|h|+|s|)(h-s).\]
\end{lemma}
\begin{proof}
A standard computation gives that 
\begin{align*}
\frac{\partial}{\partial s}(\langle\nabla_{e_i}X,e_j\rangle g^{ij}) =&\langle\frac{\partial}{\partial s}\nabla_{e_i},e_j\rangle\cdot g^{ij}+\langle\nabla_{e_i},\nabla_{e_j}\nabla d\rangle\cdot g^{ij}+\langle \nabla_{e_i}X,e_j\rangle\frac{\partial }{\partial s}g^{ij}\\
=&-\ric(\nabla d,X)+\langle\nabla X,\nabla^2d\rangle-2\langle\nabla X,\nabla^2d\rangle\\
=&-\ric(\nabla d,X)-\langle\nabla X,\nabla^2d\rangle.
\end{align*}

Then since $\frac{\partial }{\partial s}X=0$, then 
\begin{align*}
&(\frac{\partial }{\partial s}\nabla X)(e_i,e_j)\\
=&\frac{\partial }{\partial s}\langle\nabla_{e_i},e_j\rangle-(\nabla X)(\pps e_i,e_j)-(\nabla X)(e_i,\pps e_j)\\
=& R(\nabla d,e_i,X,e_j)+\langle\nabla_{e_i}X,\nabla_{e_j}\nabla d\rangle-(\nabla X)(\pps e_i,e_j)-\langle\nabla_{e_i} X,\pps e_j\rangle\\
=&R(\nabla d,e_i,X,e_j)-\langle\nabla X,A_s\rangle,
\end{align*}
where $A_s$ is the second fundamental form of $\mc N_s$. It follows that 
\[\pps|\nabla X|_{\mc N_s}\leq a+b|\nabla X|_{\mc N_s},\]
for some constants $a, b>0$. By the standard ODE inequality, we have
\[|\nabla X|_{\mc N_s}\leq C(|\nabla X|_{\mc N}+|s|).\] 
Combing with the derivative formula, we have
\begin{align*}
|\mathrm{div}_{\mc N_h}X-\mathrm{div}_{\mc N_s}X|\leq&\ C(|X|+\sup_{s\leq t\leq h}|\nabla X|_{\mc N_t})(h-s)\\
\leq&\ C(|X|+|\nabla X|_{\mc N}+|h|+|s|)(h-s).
\end{align*}
Then the desired result follows from triangle inequalities.

\end{proof}

\begin{lemma}\label{lem:minimal graph inequality}
Given $K>0$, there exist constants $\delta$ and $C$ depending only on $M,\mc N,K$ so that if $\Sigma$ and $\Gamma$ are minimal graphs with function $v,u$ on a subset of $\mc B(p;\epsilon)$ and $(v,u)$ is a $(\delta, K)$-pair, then 
\begin{align*}
|\Delta_\mc Nw|\leq &\ C|\nabla^2w|(|\nabla w|^2+|\nabla v|^2)+C|\nabla^2w|(|\nabla v||\nabla^2v|+|\nabla w|)|v|+\\
&+C(1+|\nabla^2w|+|\nabla^2v|)w+C|\nabla w||v|+C|\nabla w|\cdot|\nabla^2v|,
\end{align*}
where $w=v-u$.
\end{lemma}

\begin{proof}
For $p\in\Sigma\cap \mc N_s$, let $\{e_i\}$ be an orthonormal base of $T_p\mc N_s$. Then $\{e_i+\langle \nabla u,e_i\rangle\nabla d\}$ is a base of $T_p\Sigma$. It follows that $\n_\Sigma=(\nabla d-\nabla u)/\sqrt{1+|\nabla u|^2}$ is the unit normal vector field of $\Sigma$. Since $\Sigma$ is minimal, we have 
\[0=\mathrm{div}_{\Sigma}\n_\Sigma.\]
Naturally, such $\n_\Sigma$ can be extend to a neighborhood of $\mc N$ by parallel moving. Denote by $\wti \n$ the extended vector filed. Then we have
\[\mathrm{div}_M\wti\n|_\Sigma=\mathrm{div}_{\Sigma}\widetilde \n=0.\]

Denote by $\m$ the unit normal vector of $\Gamma$ and $\wti \m$ the extended vector field around $\mc N$. Then we also have
\[\mathrm{div}_M\wti\m|_{\Gamma}=\mathrm{div}_{\Gamma}\m=0.\]

Take $x\in\Gamma$ and $y\in\Sigma$ so that $\pi(x)=\pi(y)\in\mc N$. Set $t=d(x)$ and $s=d(y)$. Then by Lemma \ref{lem:diff of div},
\begin{equation}\label{eq:diff from t to s}
|\mathrm{div}_{\mc N_t}\wti \n-\mathrm{div}_{\mc N_s}\wti \n|\leq C(1+|\nabla \wti\n||_{\mc N})(t-s).
\end{equation} 
Similarly,
\begin{equation}\label{eq:diff from t to N}
|\mathrm{div}_{\mc N_t}(\wti\m-\wti\n)-\mathrm{div}_{\mc N}(\wti\m-\wti \n)|\leq C(|\wti \m-\wti \n|_\mc N+|\nabla(\wti\m-\wti\n)|_{\mc N})|t|.
\end{equation}
Since $\mc N$ is also minimal, hence we have
\begin{align*}
\dv_\mc N(\wti \m-\wti \n)=&\ \dv_\mc N(\frac{\nabla u}{\sqrt{1+|\nabla u|^2}}-\frac{\nabla v}{\sqrt{1+|\nabla v|^2}})\\
=&\ \dv_\mc N\big[\frac{-\nabla w}{\sqrt{1+|\nabla u|^2}}+\frac{\langle w,\nabla u+\nabla v\rangle\cdot\nabla v}{(1+|\nabla u|^2)\sqrt{1+|\nabla v|^2}+(1+|\nabla v|^2)\sqrt{1+|\nabla u|^2}}\big]
\end{align*}

Then a direct computation together with (\ref{eq:diff from t to s}) and (\ref{eq:diff from t to N}) gives that
\begin{align*}
\Big|\dv_\mc N\frac{\nabla w}{\sqrt{1+|\nabla u|^2}}\Big|\leq&\ C(1+|\nabla \wti\n||_{\mc N})w+C(|\wti \m-\wti\n|_\mc N+|\nabla(\wti\m-\wti\n)|_{\mc N})|t|+\\
&+\bigg|\dv_\mc N\frac{\langle \nabla w,\nabla u+\nabla v\rangle\cdot\nabla v}{(1+|\nabla u|^2)\sqrt{1+|\nabla v|^2}+(1+|\nabla v|^2)\sqrt{1+|\nabla u|^2}}\bigg|\\
\leq&\  C(1+|\nabla \wti\n||_{\mc N})w+C(|\nabla w|_\mc N+|\nabla(\wti\m-\wti\n)|_{\mc N})|t|+\\
&+|\nabla w|\cdot|\nabla v|\cdot |\nabla^2(u+v)|+|\nabla^2w|\cdot|\nabla(u+v)|\cdot|\nabla v|+\\
&+|\nabla w|\cdot|\nabla(u+v)|\cdot(|\nabla^2v|+|\nabla v|\cdot|\nabla^2u|\cdot|\nabla u|+|\nabla^2u|\cdot|\nabla v|^2)\\
\leq&\ C(1+|\nabla^2w|+|\nabla^2v|)w+C(|\nabla w|_\mc N+|\nabla(\wti\m-\wti\n)|_{\mc N})|t|+\\
&+C(|\nabla^2w|\cdot|\nabla v|+|\nabla^2v|\cdot|\nabla w|)(|\nabla w|+|\nabla v|)
\end{align*}
By definition,
\begin{align*}
(\wti \m-\wti\n)_\mc N=&\ \frac{\nabla d-\nabla v}{\sqrt{1+|\nabla v|^2}}-\frac{\nabla d-\nabla u}{\sqrt{1+|\nabla u|^2}}\\
=&\ \frac{\nabla d-\nabla v}{\sqrt{1+|\nabla v|^2}}-\frac{\nabla d-\nabla v}{\sqrt{1+|\nabla u|^2}}-\frac{\nabla w}{\sqrt{1+|\nabla u|^2}}\\
=&\ \frac{-\langle \nabla w,\nabla u+\nabla v\rangle\cdot(\nabla d-\nabla v)}{(1+|\nabla u|^2)\sqrt{1+|\nabla v|^2}+(1+|\nabla v|^2)\sqrt{1+|\nabla u|^2}}+\\
&+\frac{-\nabla w}{\sqrt{1+|\nabla u|^2}}.
\end{align*}
Therefore, we have
\[|(\wti \m-\wti\n)_\mc N|\leq |\nabla w|(1+|\nabla (u+v)|),\]
and 
\begin{align*}
|\nabla(\wti \m-\wti \n)|_\mc N\leq&\ |\nabla^2 w|\cdot |\nabla(u+v)|+|\nabla w|\cdot |\nabla^2(u+v)|+\\
&+|\nabla w|\cdot|\nabla(u+v)|\cdot(|\nabla^2d|+|\nabla^2v|)+\\
&+|\nabla w|\cdot |\nabla(u+v)|\cdot(|\nabla^2u|\cdot|\nabla u|+|\nabla^2v|\cdot|\nabla v|)\\
\leq &\ C\big[(|\nabla w|+|\nabla^2w|+|\nabla w|^3)(|\nabla v|+|\nabla^2v|)+|\nabla w|\cdot |\nabla^2w|+\\
&+(|\nabla w|+|\nabla w|\cdot|\nabla^2w|)(|\nabla v|+|\nabla ^2v|)^2+|\nabla^2w|\cdot|\nabla w|^3+\\
&+|\nabla w|^2+|\nabla w|\cdot|\nabla v|^2\cdot|\nabla^2v|\big]\\
\leq&\ C(|\nabla w|+|\nabla^2w|)(|\nabla v|+|\nabla^2v|+|\nabla w|).
\end{align*}
Taking them back, the inequality becomes
\begin{align*}
|\Delta w|\leq&\ C\Big|\dv_\mc N\frac{\nabla w}{\sqrt{1+|\nabla u|^2}}\Big|+C|\nabla w|\cdot|\nabla^2u|\cdot|\nabla u|\\
\leq&\ C|\nabla w|\cdot(|\nabla^2w|+|\nabla^2v|)(|\nabla w|+|\nabla v|)+\\
&+C(1+|\nabla^2w|+|\nabla^2v|)w+C|\nabla w|t+\\
&+C(|\nabla^2w|\cdot|\nabla v|+|\nabla^2v|\cdot|\nabla w|)(|\nabla w|+|\nabla v|)+\\
&+C(|\nabla w|+|\nabla^2w|)(|\nabla v|+|\nabla^2v|+|\nabla w|)t\\
\leq&\ C|\nabla^2w|(|\nabla w|^2+|\nabla v|^2)+C|\nabla^2w|(|\nabla v||\nabla^2v|+|\nabla w|)t+\\
&+C(1+|\nabla^2w|+|\nabla^2v|)w+C|\nabla w|t+C|\nabla w|\cdot|\nabla^2v|,
\end{align*}
which is exactly the desired result.
\end{proof}

\begin{lemma}\label{lem:mp from pde}
Let $\mc N$ be an $n$-dimensional Riemannian manifold with $n\geq 2$. There exist constants $C,\epsilon>0$ so that if $w>0$ and $R^2|\Delta w|\leq w/4$ on $\mc A(p;\rho,R)\subset \mc N$ for $\rho<R<\epsilon$, then
\[\min_{\partial \mc A(p;\rho,R)}w\leq C\min_{\mc A(p;\rho,R)}w,\ \  \max_{ \mc A(p;\rho,R)}w\leq C\max_{\partial \mc A(p;\rho,R)}w.\]
\end{lemma}
\begin{proof}
Denote $\rho(x)=\dist_\mc N(x,p)$. Set $w_1=e^{cr}w$ and $w_2=e^{br^2}w$. Then a direct computation gives that 
\begin{align*}
\Delta w_1=&\ w\Delta e^{cr}+e^{cr}\Delta w+2c\langle\nabla r,e^{cr}\nabla w\rangle\\
=&\ 2c\langle \nabla r,\nabla w_1\rangle+(\Delta w+(c\Delta r- c^2)w)e^{cr}
\end{align*}
Take $c=-1/(2R)$. Then the inequality becomes 
\[\Delta w_1+2\sqrt K\langle\nabla r,\nabla w_1\rangle/R\leq 0.\]
By virtue of \cite{GT}*{Theorem 8.1}, 
\[\min_{\partial \mc A(p;r,R)}w_1= \min_{\mc A(p;r,R)}w_1,\]
which implies the desired inequality.

We can also compute the following directly:
\begin{align*}
\Delta w_2=&\ w\Delta e^{br^2}+e^{br^2}\Delta w+4br\langle\nabla r,e^{br^2}\nabla w\rangle\\
=&\ 4br\langle \nabla r,\nabla w_2\rangle+(\Delta w+(2br\Delta r+2b-4b^2r^2)w)e^{br^2}.
\end{align*}
Take $b=1/(4R^2)$, then
\[\Delta w_2-4br\langle \nabla r,\nabla w_2\rangle\geq (\Delta w+w/(4R^2))e^{br^2}\geq 0.\]
Using \cite{GT}*{Theorem 8.1} again, we have
\[\max_{\partial \mc A(p;r,R)}w_2= \max_{\mc A(p;r,R)}w_2,\]
which implies the desired inequality.
\end{proof}

\section{Proof of Lemma \ref{lem:der estimate}}\label{sec:appendix:der lemma}
\begin{proof}[Proof of Lemma \ref{lem:der estimate}]
Denote by $A_\Sigma(x)$ and $A_\Gamma(x)$ the second fundamental form of $\Sigma$ and $\Gamma$ in $M$. Then a standard blowup argument gives that
\begin{equation}\label{eq:der estimate:2nd form}
\sup_{x\in\mc B(q;r)\setminus V}\dist_\mc N(x,\partial\mc B(q;r)\cup V)\cdot(|A_\Gamma(x)|+|A_\Sigma(x)|)<C.
\end{equation}

\underline{Now we prove the first inequality.} Suppose not, there exist $K>0$ and sequence of $r_j>0$, $p_j\in\mc B(p;1-r_j)$ and $V_j\subset \mc B(p;1)$ so that $\{\Sigma_j\}$ and $\{\Gamma_j\}$ are two sequences of minimal graphs over $\mc B(p_j;r_j)\setminus V_j$ with positive graph function $u_j,v_j$ satisfying
\[u_j(x)-v_j(x)>0,\ \ |u_j(x)|+|v_j(x)|<1/j,\ \ |\nabla u_j(x)|+|\nabla v_j(x)|<K , \]
and
\[\sup_{x\in \mc B(p_j;r_j)\setminus V_j}\dist_\mc N(x,\partial\mc B(p_j;r_j)\cup V_j)\cdot|\nabla\log (u_j-v_j)(x)|>j.\]
Take $q_j\in\mc B(p_j;r_j)\setminus V_j$ so that 
\begin{align*}
&\dist_\mc N(q_j,\partial\mc B(p_j;r_j)\cup V_j)\cdot|\nabla\log (u_j-v_j)(q_j)|\\
&=\sup_{x\in \mc B(p_j;r_j)\setminus V_j}\dist_\mc N(x,\partial\mc B(p_j;r_j)\cup V_j)\cdot|\nabla\log (u_j-v_j)(x)|.
\end{align*}
Set 
\[\lambda_j=|\nabla\log(u_j-v_j)(q_j)| \text{\ \   and\ \    } \rho_j=\frac{1}{2}\cdot\dist_\mc N(q_j,\partial\mc B(p_j;r_j)\cup V_j).\]
Take $q_j'\in\Sigma_j$ so that its projection to $N$ is $q_j$. Then by (\ref{eq:der estimate:2nd form}), $(\mc B(q_j';\rho_j),\lambda^2_jg,q_j')$ locally smoothly converges to a minimal graph over a hyperplane. Then it is a hyperplane by Bernstein theorem.

Now denote by $\wti \nabla$ the Levi-Civita connection under the metric $\lambda_j^2g$. Denote by $\wti u_j,\wti v_j$ the graph functions of $(\Sigma_j,\lambda_j^2g)$ and $(\Gamma_j,\lambda_j^2g)$ over $(\mc B(q_j;\rho_j),\lambda_j^2g)$. Then we have
\begin{equation}\label{eq:der estimate:size on q_j}
\frac{|\wti\nabla(\wti u_j-\wti v_j)(q_j)|_{\lambda_j^2g}}{(\wti u-\wti v_j)(q_j)}=\frac{|\nabla(u_j-v_j)(q_j)|_g}{\lambda_j(u_j-v_j)(q_j)}=1,
\end{equation}
and for any $x\in\mc B(q_j;\rho_j)$,
\[
\frac{|\wti\nabla(\wti u_j-\wti v_j)(x)|_{\lambda_j^2g}}{(\wti u_j-\wti v_j)(x)} =\frac{|\nabla (u_j-v_j)(x)|_g}{\lambda_j(u_j-v_j)(x)}<2.
\]

\begin{claim}\label{claim:uj goes to 0}
$(\wti u_j-\wti v_j)(q_j)\rightarrow 0$.
\end{claim}
\begin{proof}[Proof of Claim \ref{claim:uj goes to 0}]
Note that 
\[|\wti\nabla(\wti u_j-\wti v_j)(q_j)|_{\lambda_j^2g}=|\nabla(u_j-v_j)(q_j)|_g.\]
Then by our assumptions, it is bounded from above by $K$. Then using (\ref{eq:der estimate:size on q_j}), we have that $(\wti u-\wti v_j)(q_j)$ is bounded from above by $K$.

Recall that (\ref{eq:der estimate:2nd form}) implies that $(\Sigma_j,\lambda_j^2g,q_j')$ and $(\Gamma_j,\lambda_j^2g,q_j')$ locally smoothly converge to hyperplanes. By the assumption of $u_j-v_j>0$, such two limit hyperplanes are paralleling to each other. This deduces that $|\wti\nabla(\wti u_j-\wti v_j)(q_j)|\rightarrow 0$. Using (\ref{eq:der estimate:size on q_j}) again, we conclude that $(\wti u_j-\wti v_j)(q_j)\rightarrow 0$.
\end{proof}

Set $h_j(x)=(\wti u_j-\wti v_j)(x)/(\wti u_j-\wti v_j)(q_j)$. Then $h_j$ converges to a positive harmonic function of $\mb R^n$. Hence it is a constant. On the other hand, 
\[|\wti \nabla h_j(q_j)|=\frac{|\wti \nabla(\wti u_j-\wti v_j)(q_j)|}{(\wti u_j-\wti v_j)(q_j)}=1,\]
which leads to a contradiction.

\underline{We now prove the second inequality.} Similarly, suppose not, there exist $K>0$ and sequence of $r_j>0$, $p_j\in\mc B(p;1-r_j)$ and $V_j\subset \mc B(p;1)$ so that $\{\Sigma_j\}$ and $\{\Gamma_j\}$ are two sequences of minimal graphs over $\mc B(p_j;r_j)\setminus V_j$ with positive graph function $u_j,v_j$ satisfying
\[u_j(x)-v_j(x)>0,\ \ |u_j(x)|+|v_j(x)|<1/j,\ \ |\nabla u_j(x)|+|\nabla v_j(x)|<K , \]
and
\[\sup_{x\in \mc B(p_j;r_j)\setminus V_j}\dist^2_\mc N(x,\partial\mc B(p_j;r_j)\cup V_j)\cdot\frac{|\nabla^2\log (u_j-v_j)(x)|}{(u_j-v_j)(x)}>j.\]
Take $q_j\in\mc B(p_j;r_j)\setminus V_j$ so that 
\begin{align*}
&\dist^2_\mc N(q_j,\partial\mc B(p_j;r_j)\cup V_j)\cdot\frac{|\nabla^2\log (u_j-v_j)(q_j)|}{(u_j-v_j)(q_j)}\\
&=\sup_{x\in \mc B(p_j;r_j)\setminus V_j}\dist^2_\mc N(x,\partial\mc B(p_j;r_j)\cup V_j)\cdot\frac{|\nabla^2\log (u_j-v_j)(x)|}{(u_j-v_j)(x)}.
\end{align*}
Set 
\[\lambda_j=\sqrt{\frac{|\nabla^2\log (u_j-v_j)(q_j)|}{(u_j-v_j)(q_j)}} \text{\ \   and\ \    } \rho_j=\frac{1}{2}\cdot\dist_\mc N(q_j,\partial\mc B(p_j;r_j)\cup V_j).\]
Take $q_j'\in\Sigma_j$ so that its projection to $N$ is $q_j$. Then by (\ref{eq:der estimate:2nd form}), $(\mc B(q_j';\rho_j),\lambda^2_jg,q_j')$ locally smoothly converges to a minimal graph over a hyperplane. Then it is a hyperplane by Bernstein theorem.

Now denote by $\wti \nabla$ the Levi-Civita connection under the metric $\lambda_j^2g$. Denote by $\wti u_j,\wti v_j$ the graph functions of $(\Sigma_j,\lambda_j^2g)$ and $(\Gamma_j,\lambda_j^2g)$ over $(\mc B(q_j;\rho_j),\lambda_j^2g)$. Then we have
\begin{equation}\label{eq:der estimate:hessian change}
\frac{|\wti\nabla^2(\wti u_j-\wti v_j)(q_j)|_{\lambda_j^2g}}{(\wti u-\wti v_j)(q_j)}=\frac{|\nabla^2(u_j-v_j)(q_j)|_g}{\lambda^2_j(u_j-v_j)(q_j)}=1,
\end{equation}

\begin{claim}\label{claim:difference goes to 0}
$(\wti u_j-\wti v_j)(q_j)\rightarrow 0$.
\end{claim}
\begin{proof}[Proof of Claim \ref{claim:difference goes to 0}]
Recall that (\ref{eq:der estimate:2nd form}) implies that $(\Sigma_j,\lambda_j^2g,q_j')$ and $(\Gamma_j,\lambda_j^2g,q_j')$ locally smoothly converge to hyperplanes. Hence $|\wti \nabla^2(\wti u_j-\wti v_j)|\rightarrow 0$. Using (\ref{eq:der estimate:hessian change}), we have $ (\wti u_j-\wti v_j)(q_j)\rightarrow 0$.
\end{proof}

It follows that $(\Sigma_j,\lambda_j^2g,q_j')$ and $(\Gamma_j,\lambda_j^2g,q_j')$ locally smoothly converge to a same hyperplanes. Set $h_j(x)=(\wti u_j-\wti v_j)(x)/(\wti u_j-\wti v_j)(q_j)$. Then $h_j$ converges to a positive harmonic function of $\mb R^n$. Hence it is a constant. On the other hand, 
\[|\wti \nabla^2 h_j(q_j)|=\frac{|\wti \nabla^2(\wti u_j-\wti v_j)(q_j)|}{(\wti u_j-\wti v_j)(q_j)}=1,\]
which leads to a contradiction.

Thus, we have prove Lemma \ref{lem:der estimate}.
\end{proof}

\section{rearrange lemma}\label{sec:appendix:rearrange}
The following lemma is used in this paper frequently. 
\begin{lemma}\label{lem:number lemma}
	Let $\{\alpha_j\}_{i=1}^I$ be a sequence of positive numbers with $\alpha_j\geq1$. Then there exist $2N(\leq 2I)$ non-negative integers $\{k_j\}_{j=1}^{2N}$ such that 
	\begin{itemize}
		\item $k_{j+1}-k_{j}\geq 3$;
		\item $k_{2j}-k_{2j-1}\leq 10I+1$.
		\item $\{\alpha_j\}_{i=1}^I\subset \bigcup_{j}[k_{2j-1}+1,k_{2j}-1]$.
	\end{itemize}
\end{lemma}
\begin{proof}
Take $k_1=\max\{\inf [\alpha_j]-1,0\}$. Then define $k_{2s}$ from $k_{2s-1}$ by
	\[k_{2s}:=\inf\{t\in\mathbb Z:t\geq k_{2s-1},\alpha_j\notin(t,t+6) \text{\ holds for all\ }j \}+2.\]
Then define $k_{2s+1}$ from $k_{2s}$ by
	\[k_{2s+1}:=\sup\{t\in\mathbb Z:t\geq k_{2s},\alpha_j\notin(k_{2s},t) \text{\ holds for all\ }j \}-1.\]
It remains to check that they satisfy all the requirements. Indeed, by the choice of $k_{2s}$, $\alpha_j\notin(k_{2s}-2,k_{2s}+4)$. Thus, $k_{2s+1}+1\geq k_{2s}+4$, which implies $k_{2s+1}-k_{2s}\geq 3$.

By the choice of $k_{2s+1}$, we can see that $[k_{2s+1}+1,k_{2s+1}+2)\cap\{\alpha_j\}\neq \emptyset$. Thus, $k_{2s+2}-2\geq k_{2s+1}+1$, which implies $k_{2s+1}-k_{2s+1}\geq 3$. Finally, note that if $k_{2s}-k_{2s-1}>10I+1$, there exists $t\in (k_{2s-1}+2, k_{2s}-6)$ so that $\alpha_j\notin(t,t+6)$ for all $j$ since $\{\alpha_j\}$ contains only $I$ numbers. This contradicts the definition of $k_{2s}$.  
\end{proof}

\section{connectedness lemma}\label{sec:appendix:connectedness}
\begin{lemma}\label{lem:enlarge interior radius}
Then there exists a constant $C_0,\epsilon>0$ so that for any $r_1,r_2$ satisfy $C_0Ir\leq r_1\leq r_2<\epsilon$, there exists a $C^1$ curve $\gamma:[0,1]\rightarrow \mc A(p;r_1,r_2)\setminus\bigcup_{j=1}^I\mc B(q_j;r)$ satisfying
\begin{itemize}
\item $\gamma(0)\in\partial \mc B(p;r_1)$ and $\gamma(1)\in\partial \mc B(p;r_2)$;
\item $\mathrm{Length}(\gamma)\leq C_0 (r_2-r_1)$;
\item $\dist(\gamma,\cup_{j=1}^IB(q_j,r))\geq r_1/(C_0I)$.
\end{itemize} 
\end{lemma}

\section{Proof of (\ref{eq:result by adding with suitable coeff}) }\label{sec:appendix:suitable coeff}
In this section, we give the proof of (\ref{eq:result by adding with suitable coeff}). Such a fundamental process has been used frequently in this paper.
\begin{proof}[Proof of (\ref{eq:result by adding with suitable coeff})]
For simplicity, we define the following notions:
\begin{gather*}
A:=\mc I(p;\epsilon r)-\gamma\mc I(p;2^{k_{2N}}\theta_1\epsilon^{3/4}r^{5/4}),\\
B:=\sum_{j=j'+1}^N\gamma^{2N-2j+1}\Big[\mc I(p;2^{k_{2j}}\theta_1\epsilon^{3/4}r^{5/4})-\gamma\mc I(p;2^{k_{2j-1}}\theta_1\epsilon^{3/4}r^{5/4}) \Big],\\
C:=\sum_{j=j'}^{N-1}\gamma^{2N-2j}\Big[\mc I(p;2^{k_{2j+1}}\theta_1\epsilon^{3/4}r^{5/4})-\gamma\mc I(p;2^{k_{2j}}\theta_1\epsilon^{3/4}r^{5/4}) \Big],\\
D:=\gamma^{2N-2j'+1}\Big[\mc I(p;2^{k_{2j'}}\theta_1\epsilon^{3/4}r^{5/4})-\gamma\mc I(y_1;\theta|y_1|)\Big].
\end{gather*}

Then we have
\begin{gather*}
A\geq\frac{\tau_0}{4^{I+2}}(\log(\epsilon/r)-\log(2^{k_{2N}}\theta_1\epsilon^{3/4}r^{5/4})),\\
B\geq \sum_{j=j'+1}^{N}\gamma^{2N-2j+1}\Big[\frac{\tau_0}{4^{I+2}}(\log (2^{k_{2j}}\theta_1\epsilon^{3/4}r^{5/4})-\log (2^{k_{2j-1}}\theta_1\epsilon^{3/4}r^{5/4})) \Big],\\
D\geq \gamma^{2N-2j'+1}\Big[\frac{\tau_0}{4^{I+2}}(\log(2^{k_{2j'}}\theta_1\epsilon^{3/4}r^{5/4})-\log(\theta_1|y_1|))\Big],
\end{gather*}
and 
\begin{align*}
C&\geq 0=\sum_{j=j'}^{N-1}\gamma^{2N-2j}\Big[\frac{\tau_0}{4^{I+2}}(\log (2^{k_{2j+1}}\theta_1\epsilon^{3/4}r^{5/4})-\log (2^{k_{2j}}\theta_1\epsilon^{3/4}r^{5/4})) \Big]+\\
&+\sum_{j=j'}^{N-1}-\gamma^{2N-2j}\Big[\frac{\tau_0}{4^{I+2}}(k_{2j+1}-k_{2j})\log 2\Big].
\end{align*}
Thus,
\begin{align*}
A+B+C+D&\geq \gamma^{2N}\frac{\tau_0}{4^{I+2}}\Big[\log(\epsilon/r)-\log(\theta_1|y_1|))\Big]-\sum_{j=j'}^{N-1}\Big[\frac{\tau_0}{4^{I+2}}(k_{2j+1}-k_{2j})\log 2\Big]\\
&\geq \gamma^{2N}\frac{\tau_0}{4^{I+2}}\Big[\log(\epsilon / r)-\log(\theta_1|y_1|)\Big]-I(10I+1)\frac{\tau_0}{4^{I+2}}\log 2.
\end{align*}
On the other hand,
\begin{align*}
A+B+C+D=\mc I(p;\epsilon r)-\gamma^{2N-2j'+2}\mc I(y_1;\theta_1|y_1|)\leq \mc I(p;\epsilon r)-\gamma^{2N}\mc I(y_1;\theta_1|y_1|)/4^{I+2}.
\end{align*}
Then (\ref{eq:result by adding with suitable coeff}) follows by setting $c_0=\gamma^{2N}/r^{I+2}$ and $c_1=\gamma^{-2N}I(10I+1)\log 2$.
\end{proof}

\bibliographystyle{amsalpha}
\end{document}